\documentclass[12pt, twoside]{article}
\usepackage{enumitem}

\usepackage{tocloft}
\renewcommand{\cftsecleader}{\cftdotfill{\cftdotsep}}
\usepackage{titletoc}
\usepackage{pdfpages}
\usepackage{amsmath,amssymb}
\usepackage{amsthm}
\usepackage{leftidx}
\usepackage{epsfig}
\usepackage{graphics}
\usepackage{graphicx}
\usepackage{colordvi}
\usepackage{mathrsfs} 
\usepackage{stmaryrd}
\usepackage{geometry}
\usepackage{dsfont}
\usepackage{tikz}
\usetikzlibrary{automata}
\usetikzlibrary{snakes}
\usetikzlibrary{decorations.pathmorphing}
\usepackage{url}
\usepackage[hidelinks]
{hyperref}
\hypersetup{
    colorlinks=false,
   linkcolor=blue!75!black,
    filecolor=blue!75!black,      
   urlcolor=blue!75!black, 
   citecolor=blue!75!black,
}
\usepackage{enumitem}
\usepackage[textsize=small]{todonotes}
\usepackage[margin=1cm, size=small]{caption}
\usepackage{multirow}
\allowdisplaybreaks[4]
\oddsidemargin
-.25in \evensidemargin -.25in\marginparwidth 0.4in
\usepackage{color,soul}
\usepackage{xcolor}

\usepackage{etoolbox}
\usepackage{fancyvrb}
\usepackage{fvextra}
\usepackage{fancyhdr}
\usepackage{todonotes}

\numberwithin{figure}{section}

\usepackage{empheq}
\usetikzlibrary{automata}
\definecolor{blue2}{cmyk}{.94,.11,0,0}
\usepackage{enumitem}
\setitemize{itemsep=0pt}
\definecolor{myblue}{rgb}{.8, .8, 1}
\newlength\mytemplen
\newsavebox\mytempbox
\usepackage{verbatim}

\DeclareFontFamily{U}{mathx}{}
\DeclareFontShape{U}{mathx}{m}{n}{<-> mathx10}{}
\DeclareSymbolFont{mathx}{U}{mathx}{m}{n}
\DeclareMathAccent{\widecheck}{0}{mathx}{"71}
\DeclareMathAlphabet\mathbfcal{OMS}{cmsy}{b}{n}

\makeatletter
\renewenvironment{thebibliography}[1]
     {\section*{\refname}
      \@mkboth{\MakeUppercase\refname}{\MakeUppercase\refname}
      \begin{enumerate}[label={[\arabic{enumi}]},itemindent=*,leftmargin=2.5em]
      \@openbib@code
      \sloppy
      \clubpenalty4000
      \@clubpenalty \clubpenalty
      \widowpenalty4000
      \sfcode`\.\@m}
     {\def\@noitemerr
       {\@latex@warning{Empty `thebibliography' environment}}
      \end{enumerate}}
\makeatother

\makeatletter
\newcommand{\newparallel}{\mathrel{\mathpalette\new@parallel\relax}}
\newcommand{\new@parallel}[2]{
  \begingroup
  \sbox\z@{$#1T$}
  \resizebox{!}{\ht\z@}{\raisebox{\depth}{$\m@th#1/\mkern-4.5mu/$}}
  \endgroup
}
\makeatother

\makeatletter
\newcommand{\nnewparallel}{\mathrel{\mathpalette\nnew@parallel\relax}}
\newcommand{\nnew@parallel}[2]{
  \begingroup
  \sbox\z@{$#1T$}
  \resizebox{!}{\ht\z@}{\raisebox{\depth}{$\m@th#1/\mkern-4.5mu/\!\!\!\!\backslash $}}
  \endgroup
}
\makeatother

\makeatletter
\newcommand\cb{
    \@ifnextchar[
       {\@cb}
       {\@cb[5pt]}}

\def\@cb[#1]{
    \@ifnextchar[
       {\@@cb[#1]}
       {\@@cb[#1][5pt]}}

\def\@@cb[#1][#2]#3{
    \sbox\mytempbox{#3}
    \mytemplen\ht\mytempbox
    \advance\mytemplen #1\relax
    \ht\mytempbox\mytemplen
    \mytemplen\dp\mytempbox
    \advance\mytemplen #2\relax
    \dp\mytempbox\mytemplen
    \colorbox{myblue}{\hspace{1em}\usebox{\mytempbox}\hspace{1em}}}
\makeatother

\newcommand{\less}{\lesssim}
\newcommand{\more}{\gtrsim}
\newcommand{\e}{{\mathrm e}}
\newcommand{\1}{\mathds 1}

\newcommand{\C}{{\mathscr C}}
\newcommand{\F}{\mathscr F}
\newcommand{\B}{\mathscr B}

\newcommand{\vep}{{\varepsilon}}

\newcommand{\da}{{\downarrow}}
\newcommand{\ua}{{\uparrow}}

\newcommand{\bs}{\boldsymbol}
\newcommand{\ms}{\mathscr}

\renewcommand{\P}{{\mathbb P}}
\newcommand{\E}{{\mathbb E}}

\newcommand{\mc}{\mathcal}
\renewcommand{\d}{{\mathrm d}}

\newcommand{\R}{{\Bbb R}}

\newcommand{\s}{{\mathfrak s}}

\renewcommand{\i}{{\mathtt i}}
\newcommand{\defeq}{{\stackrel{\rm def}{=}}}

\newcommand{\EN}{{\mathcal E_N}}
\renewenvironment{proof}[1][\proofname]{\noindent {\bfseries #1.}\;}{\hfill\ensuremath{\blacksquare}\\}

\newcommand{\EM}{\gamma_{\mathsf E\mathsf M}}
\newcommand{\cc}{{^\circ}}

\newcommand{\bi}{\mathbf i}
\newcommand{\bj}{\mathbf j}
\newcommand{\bk}{\mathbf k}

\newcommand{\hK}{{\widehat{K}}}
\newcommand{\two}{{\sqrt{2}}}

\newcommand{\bbeta}{{\bs \beta}}

\renewcommand{\sp}{{\shortparallel}}
\newcommand{\nsp}{{\nshortparallel}}
\newcommand{\CN}{{\mathbb C}^{\it N}}

\newcommand{\interact}{\leftrightsquigarrow}
\newcommand{\Tw}{T^{\mathcal E_N}}

\newcommand{\Twi}
{T^{\bi\complement}}

\newcommand{\TwJm}
{T^{\bs J_m\complement}}
\newcommand{\CNw}{{\CN_{\,\sp}}}
\newcommand{\CNwi}{{\CN_{\bi\,\nsp,\,\bi\complement\,\sp}}}
\newcommand{\CNwia}{{\CN_{\bi_1\,\nsp,\,\bi_1\complement\,\sp}}}
\newcommand{\CNwib}{{\CN_{\bi_{m}\,\nsp,\,\bi_m\complement\,\sp}}}

\newcommand{\CNwid}{{\CN_{\bi_{\ell+1}\,\nsp,\,\bi_{\ell+1}\complement\,\sp}}}
\newcommand{\CNwie}{{\CN_{\bi_{\ell}\,\nsp,\,\bi_{\ell}\complement\,\sp}}}
\newcommand{\CNwj}{{\CN_{\bj\,\nsp,\,\bj\complement\,\sp}}}
\newcommand{\CNwJm}{{\CN_{\bs J_m\,\nsp,\,\bs J_m\complement\,\sp}}}
\newcommand{\CNwni}{{\CN_{\bi\complement\,\sp}}}
\usepackage{currfile}
\usepackage{fancyhdr}
\newcommand{\wt}{\widetilde}

\newcommand{\bb}{\backslash\!\backslash}
\newcommand{\bbb}{\backslash\!\backslash\!\backslash}
\patchcmd{\thebibliography}{\section*}{\section}{}{}

\newtheoremstyle{slantthm}{10pt}{10pt}{\slshape}{}{\bfseries}{}{.5em}{\thmname{#1}\thmnumber{ #2}\thmnote{ (#3)}.}
\newtheoremstyle{slantrmk}{10pt}{10pt}{\rmfamily}{}{\bfseries}{}{.5em}{\thmname{#1}\thmnumber{ #2}\thmnote{ (#3)}.}

\begin{document}
\theoremstyle{slantthm}
\newtheorem{thm}{Theorem}[section]
\newtheorem{prop}[thm]{Proposition}
\newtheorem{lem}[thm]{Lemma}
\newtheorem{cor}[thm]{Corollary}
\newtheorem{defi}[thm]{Definition}
\newtheorem{claim}[thm]{Claim}
\newtheorem{disc}[thm]{Discussion}
\newtheorem{conj}[thm]{Conjecture}

\theoremstyle{slantrmk}
\newtheorem{ass}[thm]{Assumption}
\newtheorem{rmk}[thm]{Remark}
\newtheorem{eg}[thm]{Example}
\newtheorem{question}[thm]{Question}
\numberwithin{equation}{section}
\newtheorem{quest}[thm]{Quest}
\newtheorem{problem}[thm]{Problem}
\newtheorem{discussion}[thm]{Discussion}
\newtheorem{notation}[thm]{Notation}
\newtheorem{observation}[thm]{Observation}

\definecolor{db}{RGB}{13,60,150}
\definecolor{dg}{RGB}{150,40,40}

\newcommand{\thetitle}{
\vspace{-1cm}
Stochastic motions of the two-dimensional many-body delta-Bose gas, III: Path integrals}

\title{\bf \thetitle\footnote{Support from an NSERC Discovery grant is gratefully acknowledged.}}

\author{Yu-Ting Chen\,\footnote{Department of Mathematics and Statistics, University of Victoria, British Columbia, Canada.}\,\,\footnote{Email: \url{chenyuting@uvic.ca}}}

\date{\today}

\maketitle
\abstract{This paper is the third in a series devoted to constructing stochastic motions for the two-dimensional $N$-body delta-Bose gas for all integers $N\geq 3$ and establishing the associated Feynman--Kac-type formulas. The main results here prove the Feynman--Kac-type formulas by using the stochastic many-$\delta$ motions from \cite{C:SDBG2-3} as 
the underlying diffusions. The associated multiplicative functionals show a new form and are derived from the analytic solutions of the two-dimensional $N$-body delta-Bose gas obtained in \cite{C:DBG-3}. For completeness, the main theorem includes the formula for $N=2$, which is a minor modification of the Feynman--Kac-type formula proven in \cite{C:BES-3} for the relative motions. 
}\medskip 

\noindent \emph{Keywords:} Delta-Bose gas; Schr\"odinger operators; diagrammatic expansions; path integrals; Bessel processes. \medskip

\noindent \emph{Mathematics Subject Classification (2020):} 60J55, 60J65, 60H30, 81S40.
\medskip 

\setlength{\cftbeforesecskip}{1pt}
\setlength\cftaftertoctitleskip{2pt}
\renewcommand{\cftsecleader}{\cftdotfill{\cftdotsep}}
\tableofcontents

\section{Introduction}\label{sec:intro}
This paper is the third in a series devoted to constructing stochastic motions for the two-dimensional $N$-body delta-Bose gas for all integers $N\geq 3$ and establishing the associated Feynman--Kac-type formulas. We will prove the formulas in this paper by using a $\CN$-valued strong Markov process $\{\ms Z_t\}=\{Z^j_t\}_{1\leq j\leq N}$, called a {\bf stochastic many-$\bs \delta$ motion}. It has continuous paths before its lifetime $T_\partial$ and is defined under the probability measure $\P^{\bbeta,\bs w}$ constructed in \cite{C:SDBG2-3}. Here, $\bbeta,\bs w\in (0,\infty)^{\mc E_N}$ are parameters tuning strengths of interactions of particles in different manners and play the roles of coupling constants and weights, respectively, where
\begin{align}\label{def:EN-3}
\mathcal E_N\,\defeq\,\{\bj=(j\prime,j)\in \Bbb N^2;1\leq j<j\prime\leq N\}.
\end{align}
The main properties of this stochastic many-$\delta$ motion can be found in \cite[Theorem~3.1, Proposition~3.11 and Proposition~3.15]{C:SDBG2-3}, and some of them will be recalled later. Note that  {\bf  we only consider $w_\bj>0$ for all $\bj\in \mc E_N$ in this paper} although this condition is not required in \cite{C:SDBG2-3}. Consequently, the notations here differ slightly; see \cite[Remark~3.2]{C:SDBG2-3}. Also, \cite{C:SDBG2-3} includes a description of this stochastic many-$\delta$ motion by using an SDE with singular drift explicitly expressible in $(\bbeta,\bs w)$, but
the SDE will be used here only in Remark~\ref{rmk:weight}.

The primary applications of the above stochastic many-$\delta$ motion in this paper are based on its interpretation as a family of $N$ independent two-dimensional Brownian motions conditioned for contact interactions. These interactions are realized in the form of $Z_t^{j\prime}-Z_t^{j}=0$ for $\bj=(j\prime,j)\in \mc E_N$ to the degree of nonzero local times $\{L^{\bj}_t\}$ and satisfy the {\bf no-triple-contacts (NTC)} that at most two particles can engage in contact interactions at a time. Moreover, the sample paths of the stochastic many-$\delta$ motion can be subdivided by an increasing sequence of stopping times $\{T^m_0\}_{m\in \Bbb N}$, called the {\bf contact-creation times}, according to the different pairs of particles that can engage in contact interactions. Specifically, in $[T_0^{m},T_0^{m+1})$ for any $m\geq 0$ ($T_0^0\equiv 0$), no contact interactions occur in particles with different states at time $T_0^{m}$, and whenever $m\geq 1$, contact interactions occur in a unique pair of particles, labelled by a random index $\bs J_m=(J_m\prime,J_m)\in \mc E_N$. These characteristics make possible the construction of the stochastic many-$\delta$ motion by locally transforming a {\bf stochastic one-$\bs \delta$ motion} from \cite{C:BES-3,C:SDBG1-3} with Girsanov's theorem over $[T_0^{m},T_0^{m+1})$. Nonetheless, the continuous extension to $T_0^{m+1}$ concerns constructing and studying the exit measure of a supermartingale. New properties are {elicited} from the extremal behavior of the stochastic one-$\delta$ motion for changing measures, including especially the NTC at $T_0^{m+1}$ to restart the construction over the next interval $[T_0^{m+1},T_0^{m+2})$. 

\subsection{Analytic path integrals}\label{sec:intro-analytic}
The Feynman--Kac-type formulas in this paper represent limiting semigroups 
of the following approximate Hamiltonians for the two-dimensional $N$-body delta-Bose gas:
\begin{align}\label{def:HN-3}
 -\frac{1}{2}\sum_{j=1}^N\Delta_{z^j}-\sum_{\bj=(j\prime,j)\in \mathcal E_N}\left(\frac{2\pi}{\log \vep^{-1}}+\frac{2\pi \lambda_\bj}{\log^{2} \vep^{-1}}\right) \vep^{-2}\phi_\bj(\vep^{-1}(z^{j\prime}-z^j)),\quad z^j\in \Bbb C,
\end{align}
where $\lambda_\bj$  and $\phi_\bj$  can use any given real constants and probability densities in $ \C_c(\R^2)$, respectively. Note that these approximations show a renormalization generalizing the two-particle case introduced in \cite{AGHH:2D-3}; see \cite[Section~1]{C:SDBG1-3} for a discussion of this setting. 

Specifically, we will work with the analytic solutions $Q^{\bbeta;N}_{0;t}f(z_0)$ obtained as the limits of the following Feynman--Kac semigroups representing the approximate Hamiltonians in \eqref{def:HN-3}: 
\begin{align}\label{DBG:approx}
&\quad\; \lim_{\vep\to 0}\E^{(0)}_{z_0}\Biggl[\exp\Biggl\{
\sum_{\bj=(j\prime,j)\in \mathcal E_N}\int_0^t\left(\frac{2\pi}{\log \vep^{-1}}+\frac{2\pi \lambda_\bj}{\log^{2} \vep^{-1}}\right)\vep^{-2} \phi_\bj(\vep^{-1}(Z^{j\prime}_r-Z^j_r))\d r\Biggl\}f(\ms Z_t)\Biggr]\notag\\
&=Q^{\bbeta;N}_{0;t}f(z_0),\quad\forall\; f\in \B_b(\CN).
\end{align}
Here, $z_0=(z_0^1,\cdots,z_0^N)\in \CN$ satisfies $z_0^{j\prime}\neq z_0^j$ for all $j\prime\neq j$, 
 $\{\ms Z_t\}=\{Z^j_t\}_{1\leq j\leq N}$ under $\P_{z_0}^{(0)}$ is a $2N$-dimensional standard Brownian motion with initial condition $z_0$, and 
\begin{align}
Q^{\bbeta;N}_{0;t}f(z_0)\,\defeq\,{\mathbb E}^{(0)}_{z_{0}}
\bigl[f(\ms Z_{t})\bigr]+ \sum_{m=1}^{\infty }
\sum_{
\substack{ {\mathbf i}_{1},\ldots ,{\mathbf i}_{m}\in \mathcal E_{N}\\{\mathbf i}_{1}\neq \cdots \neq {\mathbf i}_{m}}}
\int _{0< s_{1}< \cdots < s_{m}< t} P^{\bbeta;{\mathbf i}_{1},\ldots ,{
\mathbf i}_{m}}_{s_{1},\ldots ,s_{m},t}f(z_{0})
\mathrm{d} \boldsymbol s_{m}. \label{P:series}  
\end{align}
The definition in \eqref{P:series} uses $\d \bs s_m\,\defeq\,\d s_1\cdots \d s_m$ and $P^{\bbeta;{\mathbf i}_{1},\ldots ,{
\mathbf i}_{m}}_{s_{1},\ldots ,s_{m},t}f(z_{0})$ from \eqref{def:Qsummands} 
as an iterated space-time Riemann integral of products of copies of the following functions:
\begin{gather}
P_{s'-s}(z,z')\,=P_{s'-s}(z'-z)\;\defeq\,
\frac{1}{2\pi (s'-s)}\exp \biggl(- \frac{ |z'-z| ^{2}}{2(s'-s)} \biggr),\label{def:Pt-3} \\
 \s^{\wt{\beta}}(\tau)\;\defeq\, 4\pi \int_0^\infty \frac{\wt{\beta}^u\tau^{u-1}}{\Gamma(u)}\d u.\label{def:sbeta-3}
\end{gather}
Also, the set $\bbeta$ of coupling constants has components $\beta_\bj$ determined by the following equations: 
\begin{linenomath*}\begin{align*}
\frac{\log \beta_\bj}{2}=\log 2+\lambda_{\bj}-\EM-\int_{\Bbb C}\int_{\Bbb C}\phi_\bj(z)\phi_\bj(z')\log |z-z'|\d z\d z',
\end{align*}\end{linenomath*}
where $\EM$ is the Euler--Mascheroni constant. Note that \eqref{DBG:approx} can be justified by the proof of \cite{C:DBG-3} \emph{mutatis mutandis}, although \cite{C:DBG-3}  only considers the homogeneous case of $\lambda_\bj\equiv \lambda$ and $\phi_\bj\equiv \phi$; see \cite{GQT-3} for the analytic solutions of the homogeneous case under $f\in L^2$.

\begin{figure}[t]
\begin{center}
 \begin{tikzpicture}[scale=1]
    \draw[gray, thin] (0,0) -- (7.6,0);
    \foreach \i in {0.0, 0.5, 1.2, 2.0, 2.5, 3.2, 3.7, 4.7, 5.5, 6.4, 7.0, 7.6} {\draw [gray] (\i,-.05) -- (\i,.05);}
    \draw (0.0,-0.03) node[below]{$0$};
    \draw (0.5,-0.07) node[below]{$s_1$};
    \draw (1.2,-0.07) node[below]{$\tau_1$};
    \draw (2.0,-0.07) node[below]{$s_2$};
    \draw (2.5,-0.07) node[below]{$\tau_2$};
    \draw (3.2,-0.07) node[below]{$s_3$};
    \draw (3.7,-0.07) node[below]{$\tau_3$};
    \draw (4.7,-0.07) node[below]{$s_4$};
    \draw (5.5,-0.07) node[below]{$\tau_4$};
    \draw (6.4,-0.07) node[below]{$s_5$};
    \draw (7.0,-0.07) node[below]{$\tau_5$};
    \draw (7.6,-0.03) node[below]{$t$};
    \node at (0.0,0.40) {$\bullet$};
    \node at (0.0,1.20) {$\bullet$};
    \node at (0.0,1.80) {$\bullet$};
    \node at (0.0,2.30) {$\bullet$};
    \draw (0.0,0.40) node [left] {$z_0^1$};
    \draw (0.0,1.20) node [left] {$z_0^2$};
    \draw (0.0,1.80) node [left] {$z_0^3$};
    \draw (0.0,2.30) node [left] {$z_0^4$};  
    \draw [thick, color=black] (0.0,0.40) -- (0.5,0.6);
    \draw [thick, color=black] (0.0,1.20) -- (0.5,0.6);
    \draw [thick, color=black] (0.0,1.80) -- (0.5,1.15);    
    \draw [thick, color=black] (0.0,2.30) -- (0.5,2.20); 
    \node at (0.5,0.6) {$\bullet$};
    \node at (0.5,0.6) {$\bullet$};
    \node at (0.5,1.15) {$\bullet$};
    \node at (0.5,2.20) {$\bullet$};
    \draw [thick, color=black, snake=coil, segment length=4pt] (0.5,0.60) -- (1.2,0.80);
    \draw [thick, color=black] (0.5,1.15) -- (1.2,1.40);    
    \draw [thick, color=black] (0.5,2.20) -- (1.2,2.00);    
    \node at (1.2,0.8) {$\bullet$};
    \node at (1.2,0.8) {$\bullet$};
    \node at (1.2,1.40) {$\bullet$};
    \node at (1.2,2.00) {$\bullet$};
    \draw [thick, color=black] (1.2,0.8) -- (2.0,0.15);
    \draw [thick, color=black] (1.2,0.8) -- (2.0,1.45);
    \draw [thick, color=black] (1.2,1.40) -- (2.0,1.45);    
    \draw [thick, color=black] (1.2,2.00) -- (2.0,2.20);  
    \node at (2.0,0.15) {$\bullet$};
    \node at (2.0,1.45) {$\bullet$};
    \node at (2.0,1.45) {$\bullet$};
    \node at (2.0,2.20) {$\bullet$};
    \draw [thick, color=black] (2.0,0.15) -- (2.5,0.45);
    \draw [thick, color=black, snake=coil, segment length=4pt] (2.0,1.44) -- (2.52,1.185);    
    \draw [thick, color=black] (2.0,2.20) -- (2.5,1.80);    
    \node at (2.5,0.45) {$\bullet$};
    \node at (2.5,1.2) {$\bullet$};
    \node at (2.5,1.2) {$\bullet$};
    \node at (2.5,1.80) {$\bullet$};
    \draw [thick, color=black] (2.5,0.45)--(3.2,0.85);  
    \draw [thick, color=black] (2.5,1.20)--(3.2,0.85);    
    \draw [thick, color=black] (2.5,1.20)--(3.2,1.35);  
    \draw [thick, color=black] (2.5,1.80)--(3.2,2.10);        
    \node at (3.2,0.85) {$\bullet$};
    \node at (3.2,0.85) {$\bullet$};
    \node at (3.2,1.35) {$\bullet$};
    \node at (3.2,2.10) {$\bullet$};
    \draw [thick, color=black, snake=coil, segment length=4pt] (3.2,0.85)--(3.68,0.49) ;
    \draw [thick, color=black] (3.2,1.35)--(3.7,1.05);    
    \draw [thick, color=black] (3.2,2.10)--(3.7,2.02);  
    \node at (3.7,0.5) {$\bullet$};
    \node at (3.7,0.5) {$\bullet$};
    \node at (3.7,1.05) {$\bullet$};
    \node at (3.7,2.02) {$\bullet$};
    \draw [thick, color=black] (3.7,0.5)--(4.7,1.25);
    \draw [thick, color=black] (3.7,0.5)--(4.7,1.75);    
    \draw [thick, color=black] (3.7,1.05)--(4.7,2.30);    
    \draw [thick, color=black] (3.7,2.02)--(4.7,2.30);    
    \node at (4.7,1.25) {$\bullet$};
    \node at (4.7,1.75) {$\bullet$};
    \node at (4.7,2.30) {$\bullet$};
    \node at (4.7,2.30) {$\bullet$};
    \draw [thick, color=black] (4.7,1.25)--(5.5,0.45);
    \draw [thick, color=black] (4.7,1.75)--(5.5,1.20);   
    \draw [thick, color=black, snake=coil, segment length=4pt] (4.7,2.30)--(5.5,1.80); 
    \node at (5.5,0.45) {$\bullet$};
    \node at (5.5,1.20) {$\bullet$};
    \node at (5.5,1.8) {$\bullet$};
    \node at (5.5,1.8) {$\bullet$};
    \draw [thick, color=black] (5.5,0.45)--(6.4,0.35);
    \draw [thick, color=black] (5.5,1.20)--(6.4,0.35); 
    \draw [thick, color=black] (5.5,1.8)--(6.4,1.15);    
    \draw [thick, color=black] (5.5,1.8)--(6.4,2.40); 
    \node at (6.4,0.35) {$\bullet$};
    \node at (6.4,0.35) {$\bullet$};
    \node at (6.4,1.15) {$\bullet$};
    \node at (6.4,2.40) {$\bullet$};
    \draw [thick, color=black, snake=coil, segment length=4pt] (6.4,0.35)--(7.0,0.95);
    \draw [thick, color=black] (6.4,1.15)--(7.0,1.85);    
    \draw [thick, color=black] (6.4,2.40)--(7.0,2.30);
    \node at (7.0,0.95) {$\bullet$};
    \node at (7.0,0.95) {$\bullet$};
    \node at (7.0,1.85) {$\bullet$};
    \node at (7.0,2.30) {$\bullet$};
    \draw [thick, color=black] (7.0,0.95)--(7.6,0.45);
    \draw [thick, color=black] (7.0,0.95)--(7.6,0.85);    
    \draw [thick, color=black] (7.0,1.85)--(7.6,1.55); 
    \draw [thick, color=black] (7.0,2.30)--(7.6,2.15); 
    \node at (7.6,0.45) {$\bullet$};
    \node at (7.6,0.85) {$\bullet$};
    \node at (7.6,1.55) {$\bullet$};
    \node at (7.6,2.15) {$\bullet$};
\end{tikzpicture}
\end{center}
\vspace{-.5cm}
\caption{The figure illustrates the graphical representation of $P^{\bbeta;\bi_1,\bi_2,\bi_3,\bi_4,\bi_5}_{s_1,s_2,s_3,s_4,s_5,t}f(z_0)$ defined in \cite{C:DBG-3}, with $N=4$, $\bi_1=(2,1)$, $\bi_2=(3,2)$, $\bi_3=(2,1)$, $\bi_4=(4,3)$ and $\bi_5=(2,1)$. 
}
\label{fig:2-3}
\end{figure}

Despite the notational complexity in \eqref{def:Qsummands}, each $P^{\bbeta;{\mathbf i}_{1},\ldots ,{\mathbf i}_{m}}_{s_{1},\ldots ,s_{m},t}f(z_{0})$ mentioned above has been known for simple interpretations (e.g. \cite{CSZ:Mom-3, C:DBG-3,GQT-3}). For example, the interpretation in \cite[Section~4]{C:DBG-3} uses a graphical representation suggesting paths of $N$ particles with pairwise contact interactions. Specifically, with respect to the space-time integration in \eqref{def:Qsummands}, the graphical representation from \cite{C:DBG-3} views the integrand of each kernel, excluding $f$, as a product of edge weights of two types such that every edge weight defines a line segment in space-time:
\begin{itemize}
\item The edge weight of the first type takes the form of a heat kernel $P_{s'-s}(\mbox{cnst}\cdot z,z')$ and defines a straight line segment from $(z,s)$ to $(z', s')$. This line segment represents the path of one particle, which is \emph{free} in the sense of not being affected by any contact interaction. 
\item The edge weight of the second type takes the form of a product of
$\s^{\wt{\beta}}(s'-s)P_{s'-s}(\mbox{cnst}\cdot z,z')$ and defines a coiled line segment
from $(z,s)$ and $(z',s')$. This line segment simultaneously represents the paths of \emph{two particles} undergoing contact interactions.  
\end{itemize}
The overall interpretation of the graphical representation from \cite{C:DBG-3}  thus considers a discrete graph such that the vertices are distinct pairs of the form $(\mathring{z},\mathring{s}) \in \R^2\times\R_+$. Here, $\mathring{z}$ is chosen from the space variables in \eqref{def:Qsummands} for integration (Remark~\ref{rmk:spatialv}) or is one of the distinct points $z_0^1,\cdots,z_0^N$ in the initial condition; $\mathring{s}$ is chosen from the time variables in  \eqref{def:Qsummands} or is one of the times $0$ and $t$. See Figure~\ref{fig:2-3} for an example. 

\subsection{Main results of this paper}\label{sec:intro-main}
The main results of this paper are summarized as the following theorem. It gives Feynman--Kac-type formulas 
for the analytic solutions $Q^{\bbeta;N}_{0;t}f(z_0)$ in \eqref{P:series}. In particular, for $N=2$, the formula is a minor modification of the one proven in \cite{C:BES-3} for the relative motions. 

\begin{thm}[Feynman--Kac-type formulas]\label{thm:final}
Let $N\geq 2$ be an integer, $\bbeta\in (0,\infty)^{\mc E_N}$, and $\bs w\in (0,\infty)^{\mc E_N}$. 
For any $z_0\in \CN$ with $z_0^{j\prime}-z_0^j\neq 0$ for all $\bj=(j\prime,j)\in\mc E_N$, 
\begin{align}\label{goal:FK}
Q^{\bbeta;N}_{0;t}f(z_0)=
\E^{\bbeta,\bs w}_{z_0}\Biggl[
\frac{\e^{{A}^{\bbeta,\bs w}_0(t)}\sum_{\bj\in \mc E_N}w_\bj K_0(\sqrt{2\beta_\bj}|Z^{\bj}_0|)}{\sum_{\bj\in \mc E_N}w_\bj K_0(\sqrt{2\beta_\bj}|Z^{\bj}_t|)}f(\ms Z_t);t<T_\partial\Biggr],\;\; \forall\;f\in \B_{b,+}(\CN),
\end{align}
where $K_\nu(\cdot)$ is the Macdonald function of index $\nu$, 
\begin{align}
\begin{split}
{A}^{\bbeta,\bs w}_0(t)&\,\defeq\,
\sum_{\bi,\bj\in \mc E_N, 
\bi\neq \bj}2\left(\frac{w_\bj}{w_\bi}\right)\int_0^t K_0(\sqrt{2\beta_\bj} |Z^{\bj}_s|)\d L^{\bi}_s+\int_0^t  \frac{\sum_{\bj\in \mathcal E_N} \beta_\bj w_\bj K_0(\sqrt{2\beta_\bj} |Z^{\bj}_s|)}{\sum_{\bj\in \mc E_N}w_\bj K_0(\sqrt{2\beta_\bj} |Z^{\bj}_s|)}\d s,
\label{def:finalA}
\end{split}
\end{align}
and $\{L^\bi_t\}$ is the local time of the {\bf stochastic relative motion} $\{Z^{\bi}_t\,\defeq\,(Z^{i\prime}_t-Z^i_t)/\two\}$, $\bi=(i\prime,i)\in \mc E_N$, at level zero.
\end{thm}

Note that in the formula of \eqref{goal:FK}, the weights $\bs w$ in the multiplicative functional and the stochastic many-$\delta$ motion ``cancel out'' each other; see Remark~\ref{rmk:weight} for heuristics. Also, more details of the normalization of the local times used in Theorem~\ref{thm:final} can be found in Remark~\ref{rmk:localtime}.

Theorem~\ref{thm:final} for $N\geq 3$ is established in this paper in the finer form of Theorem~\ref{thm:main2} that proves the Feynman--Kac-type formulas for each summand in \eqref{P:series}. The proof of Theorem~\ref{thm:main2}
 transforms such a summand to an iterated expectation using several stochastic one-$\delta$ motions studied in \cite{C:SDBG1-3} (Section~\ref{sec:SGP2}) and then converts the relevant stochastic one-$\delta$ motions to a single stochastic many-$\delta$ motion (Section~\ref{sec:FKproof}). In particular, the conversion addresses the technical issue that the local construction of the stochastic many-$\delta$ motion up to the next contact-creation time is closed from the right only \emph{implicitly} as an exit measure of a supermartingale, but proving the Feynman--Kac-type formulas requires quantifications of this closure (Proposition~\ref{prop:connector}). On the other hand, the overall argument of the proof of Theorem~\ref{thm:main2} considers and establishes a relationship between the above interpretation from \cite{C:DBG-3} for $P^{\bbeta;{\mathbf i}_{1},\ldots ,{\mathbf i}_{m}}_{s_{1},\ldots ,s_{m},t}f(z_{0})$ and the sample paths of the stochastic many-$\delta$ motions. Specifically, the time variables $s_1,\cdots,s_m$ correspond to the contact-creation times $T_0^1,\cdots,T_0^m$, and the indices $\bi_1,\cdots,\bi_m$ correspond to the random indices $\bs J_1,\cdots,\bs J_m$ mentioned before Section~\ref{sec:intro-analytic} for particles undergoing contact interactions. See the discussion below \eqref{m=2sum} for more details. 

Besides the method of proof, we see the Feynman--Kac-type formulas in Theorem~\ref{thm:final} for $N\geq 3$ very different from the others in the literature. For example, the present formulas include nontrivial boundary terms realized as the local-time integrals in \eqref{def:finalA} when one compares them with formulas using only ground state transformations; see \cite{C:SDBG1-3,C:SDBG2-3} for discussions. Moreover, each of these local-time integrals should be regarded as singular because $K_0(0)=\infty$ and also involves other particles via $ K_0(\sqrt{2\beta_\bj}|Z^{\bj}_s|)$ for $\bj\neq \bi$. By contrast, such involvement is not in the Feynman--Kac formula of the one-dimensional many-body delta-Bose gas~\cite{BC:1D-3}. \medskip 

\noindent {\bf Frequently used notation.} $C(T)\in(0,\infty)$ is a constant depending only on $T$ and may change from inequality to inequality unless indexed by labels of equations. Other constants are defined analogously. We write $A\less B$ or $B\more A$ if $A\leq CB$ for a universal constant $C\in (0,\infty)$. $A\asymp B$ means both $A\less B$ and $B\more A$. For a process $Y$, the expectations $\E^Y_y$ and $\E^Y_\nu$ and the probabilities $\P^Y_y$ and $\P^Y_\nu$ mean that the initial conditions of $Y$ are the point $x$ and the probability distribution $\nu$, respectively. \emph{Unless otherwise mentioned, the processes are subject to constant initial conditions}. To handle many-body dynamics, we introduce the following notation of \emph{multiplication columns}:
\begin{align}
\label{def:column-3} 
\begin{bmatrix} a_{1}
\\
\vdots
\\
a_{n} \end{bmatrix}_{\times } \stackrel{\mathrm{def}} {=}
a_{1} \times \cdots \times a_{n}. 
\end{align}
In most cases, such a multiplication column will display the product of
$a_{1},\ldots ,a_{n}$ ``occurring over the same period of time,'' and the
entries indexed by $i\prime $ and $i$ will appear in the first and second
rows, respectively. Products of measures will be denoted similarly by using
$[\cdot ]_{\otimes}$. Lastly, $\log $ is defined with base $\e$, and $\log^ba\,\defeq\, (\log a)^b$. \medskip

\noindent {\bf Frequently used asymptotic representations.} The following can be found in \cite[p.136]{Lebedev-3}:
\begin{align}
&K_0(x)\sim \log x^{-1},&&\hspace{-3cm} x\searrow 0,\label{K00-3}\\
&K_0(x)\sim \sqrt{\pi/(2x)}\e^{-x},&&\hspace{-3cm} x\nearrow\infty.\label{K0infty-3}
\end{align}

\section{Stochastic path integrals}\label{sec:path-int}
Our goal in this section is to prove Theorem~\ref{thm:main2}, which refines Theorem~\ref{thm:final}. 
Let us begin by specifying the kernels $P^{\bbeta ;{\mathbf i}_{1},\ldots ,{\mathbf i}_{m}}_{s_{1},\ldots ,s_{m},t}f(z_{0})$ discussed in Section~\ref{sec:intro} and introducing some notations for the stochastic many-$\delta$ motions $\{\ms Z_t\}$ under $\P^{\bbeta,\bs w}$ for $\bbeta,\bs w\in (0,\infty)^{\mc E_N}$ \cite{C:SDBG2-3}. First, given any $z=(z^1,\cdots,z^N)\in \CN$ and $\bi=(i\prime,i)\in \mc E_N$, set 
\begin{align}\label{def:unitary-3}
z^{\bi}\,\defeq\, \frac{z^{i\prime}-z^{i}}{\two},\quad 
z^{\bi\prime}\,\defeq\, \frac{z^{i\prime}+z^{i}}{\two},
\end{align}
and we take the following sets as state spaces or sets of eligible initial conditions: 
\begin{align}
\begin{split}\label{def:CNw-3}
\CNwi&\,\defeq\,\{z\in \CN;z^\bi=0\;\&\; z^\bj\neq 0\;\forall\;\bj\in \mc E_N\setminus\{\bi\}\},\quad \bi\in \mc E_N,\\
 \CNw&\,\defeq\, \{z\in \CN;z^\bj\neq 0\;\forall\;\bj\in \mc E_N\}.
\end{split}
\end{align}

\begin{defi}
\label{thm:DBGsgp}
For all $\bbeta\in (0,\infty)^{\mc E_N}$, $\bi_1,\cdots,\bi_m\in \mc E_N$ with $\bi_1\neq \bi_2\neq \cdots\neq \bi_m$, $f\in \B_b(\CN)$ and $z_0\in\CNw$,
\begin{align}
\begin{aligned}
 P^{\bbeta ;{\mathbf i}_{1},\ldots ,{\mathbf i}_{m}}_{s_{1},\ldots ,s_{m},t}f(z_{0})& \stackrel{
\mathrm{def}} {=}
\int _{{\mathbb C}^{N-1}}
\d z_1^{/\bi_1}
K_{0,s_1}\bigl(z_{0},z_{1}^{/{\mathbf i}_{1}}\bigr)
\\
&\quad\;
\int _{{\mathbb C}^{N-1}}\d z_2^{/\bi_2}
\int _{s_{1}}^{s_{2}} \mathrm{d}\tau _{1}
K^{\beta_{\bi_1} ;{\mathbf i}_{1}}_{s_{1},\tau _{1},s_{2}}\bigl(z_{1}^{/{
\mathbf i}_{1}},z_{2}^{/{\mathbf i}_{2}}
\bigr)\cdots
\\
&\quad\;
\int _{{\mathbb C}^{N-1}}\d z^{/\bi_m}_m
\int _{s_{m-1}}^{s_{m}} \mathrm{d}\tau _{m-1}
K^{\beta_{\bi_{m-1}} ;{\mathbf i}_{m-1}}_{s_{m-1},\tau _{m-1},s_{m}}\bigl(z^{/{
\mathbf i}_{m-1}}_{m-1},z^{/{\mathbf i}_{m}}_{m}
\bigr)
\\
&\quad\;
\int _{{\mathbb C}^{N}}\d z_{m+1}
\int _{s_{m}}^{t} \mathrm{d}\tau _{m}
K^{\beta_{\bi_m} ;{\mathbf i}_{m}}_{s_{m},\tau _{m},t}\bigl(z_{m}^{/{
\mathbf i}_{m}},z_{m+1}
\bigr)f(z_{m+1}). \label{def:Qsummands}
\end{aligned}
\end{align}
The right-hand side of \eqref{def:Qsummands} uses the following pieces of notation:
for $u_0=(u_0^1,\cdots,u_0^N)$, $v_0=(v_0^1,\cdots,v_0^N)\in \CN$ and $\wt{\beta}>0$,
\begin{align}
&K_{\tau,s'}(u_{0},v_{0})\stackrel{\mathrm{def}} {=}
\prod_{k=1}^{N}P_{s'-\tau}
\bigl(u_{0}^{k},v_{0}^{k}\bigr),
\label{def:K}
\\
& K^{\wt{\beta} ;{\mathbf i}}_{s,\tau ,s'}(u_{0},v_{0})
\stackrel{\mathrm{def}} {=}
\int _{{\mathbb C}^{N-1}} 
\d \tilde{z}^{\backslash\bi}
 K^{\wt{\beta} ;{\mathbf i};\interact}_{s,\tau }\bigl(u_{0},\tilde{z}^{\backslash \bi}\bigr) K_{\tau ,s'}\bigl(\tilde{z}^{\backslash\!\backslash\bi},v_{0}\bigr),\label{def:Kbeta}\\
& K^{\wt{\beta} ;{\mathbf i};\interact}_{s,\tau }(u_{0},v_0)\stackrel{\mathrm{def}} {=} 
\begin{bmatrix} P_{\tau -s}\bigl({\sqrt{2}}u_{0}^{i},v_0^{i\prime}
\bigr)
\\[6pt]
{\mathfrak s}^{\wt{\beta}}(\tau -s)
\\[6pt]
\,\prod_{k:k\notin {\mathbf i}}P_{\tau -s}\bigl(u_{0}^{k},v_0^{k}
\bigr) \end{bmatrix}_{\times },\label{def:Kbetainteract} \\
&u_{0}^{/{\mathbf i}}\,=\bigl(u_{0}^{1/{\mathbf i}},
\ldots ,u_{0}^{N/{\mathbf i}}\bigr)\in \CN:
u_0^{k/\bi}\stackrel{\mathrm{def}} {=}
\begin{cases}
u_0^i,&k\in\bi,\\
u_0^k,&k\notin \bi, 
\end{cases}
\label{def:uslash} \\
&u_{0}^{\backslash{\mathbf i}}\,\in \Bbb C^{\bi\complement\cup \{i\prime\}}: u_0^{k\backslash\bi}\stackrel{\mathrm{def}} {=}
\begin{cases}
u_0^{\bi\prime},& k=i\prime,\\
u_0^k,&k\notin \bi,
\end{cases}
\label{def:ubslash}\\
&u_0^{\bb\bi}\,=\bigl(u_0^{1\bb\bi},\cdots,u_0^{N\bb\bi}\bigr)\in \CN: u_0^{k\bb\bi}\stackrel{\mathrm{def}} {=}
\begin{cases}
u_0^{\bi\prime}/\two ,& k\in \bi,\\
u_0^k,&k\notin \bi
\end{cases}\label{def:udbslash},\\
\label{def:integrators}
&\d z_\ell\defeq \bigotimes_{k=1}^N \d z_\ell^k,\quad
\d z^{/\bi}_\ell \defeq \bigotimes_{k:k\neq i\prime}\d z_\ell^{k},
\quad
\d \tilde{z}^{\backslash\bi}\,\defeq \Biggl(\bigotimes
_{k:k
\notin \bi}\mathrm{d}\tilde{z}^{k}\Biggr)\otimes \d \tilde{z}^{\bi\prime}.
\end{align}
Here, $P_{s'-s}(z,z')$ and ${\mathfrak s}^{\wt{\beta}}(\cdot)$ are defined in \eqref{def:Pt-3} and \eqref{def:sbeta-3}, and $\bi=(i\prime,i)\in \mc E_N$ is also regarded as the set $\{i\prime,i\}$.
\end{defi}

\begin{rmk}\label{rmk:spatialv}
The space variables for drawing the graphical representation of $P^{\bbeta ;{\mathbf i}_{1},\ldots ,{\mathbf i}_{m}}_{s_{1},\ldots ,s_{m},t}f(z_{0})$ discussed in Section~\ref{sec:intro} are those $\R^2$-valued variables from the right-hand sides of the equations in \eqref{def:integrators}. \qed  
\end{rmk}

The other notations to be used in the statement of Theorem~\ref{thm:main2} concern the contact-creation times $T_0^m$ and the random indices $\bs J_m$ of pairs of particles for the associated contact interactions, where $m\geq 1$.
To specify the contact-creation times, first, we set 
\begin{align}
\Tw_\eta&\,\defeq\,\inf\{t\geq 0;\exists\;\bj\;\mbox{s.t.}\; |Z^\bj_t|=\eta\},\label{def:Tw-3}\\
\Twi_\eta&\,\defeq
 \inf\{t\geq 0;\;\exists\;\bj\neq \bi\;\mbox{s.t.}\;|Z^\bj_t|=\eta\}\label{def:Twi-3}.
\end{align}
Then, the contact-creation times are defined inductively as the following stopping times $T^m_0$, $m\in \Bbb N$, using the shift operators $\vartheta_t$: 
\begin{align}
\label{def:Tm1-3}
 T^1_0&\,\defeq\, T_0^{\mc E_N},
\quad T^{m+1}_0\,\defeq\, \begin{cases}
\TwJm_0\circ \vartheta_{T^m_0}, &\mbox{if}\quad T^m_0<\infty,\\
\infty, &\mbox{otherwise},
\end{cases}
\end{align}
under $\P^{\bbeta,\bs w}_{z_0}$ for $z_0\in \CNw$, and
\begin{align}
 T^{1}_0=T_0^{\bi,1}&\,\defeq\, 
\Twi_0, \quad 
T^{m+1}_0=T^{\bi,m+1}_0\,\defeq\, \begin{cases}
\TwJm_0\circ \vartheta_{T^{m}_0}, &\mbox{if}\quad T^{\bi,m}_0<\infty,\\
\infty, &\mbox{otherwise},
\end{cases}\label{def:Tm2-3}
\end{align}
under $\P^{\bbeta,\bs w}_{z_0}$ for $z_0\in \CNwi$, where $\bs J_m$ is the unique random index in $\mc E_N$ such that $\ms Z_{T^{m}_0}\in \CNwJm$. In particular, the lifetime $T_\partial$ of the stochastic many-$\delta$ motion is defined as follows: 
\begin{align}\label{def:Tinfty-3}
\begin{aligned}
T_\partial=T_0^\infty&\,\defeq\,\lim_{m\to\infty}\ua\, T^m_0. 
\end{aligned}
\end{align}
Note that $\P^{\bbeta,\bs w}_{z_0}(T_0^1<\infty)=1$ for all $z_0\in \CNw$, and if $\beta_\bj=\beta$ for all $\bj\in \mc E_N$, then $T_0^m$ is a.s. finite for all $m\in \Bbb N$ \cite[Theorem~3.1]{C:SDBG2-3}. 

\begin{thm}\label{thm:main2}
Fix an integer $N\geq 3$, $\bbeta,\bs w\in (0,\infty)^{\mc E_N}$, $z_0\in \CNw$,
and $f\in \B_{b,+}(\CN)$. Write
\begin{align}
K_0^{\bbeta,\bj}(s)\,\defeq\, K_0(\sqrt{2\beta_\bj}|Z^\bj_s|), \quad 
K^{\bbeta,\bs w}_0(s)\,\defeq\,\sum_{\bj\in \mc E_N}w_\bj K_0^{\bbeta,\bj}(s).
\end{align}
Then with respect to the stochastic many-$\delta$ motion $\{\ms Z_t\}=\{Z^j_t\}_{1\leq j\leq N}$ 
defined under $\P^{\bbeta,\bs w}_{z_0}$, the following Feynman--Kac-type formulas hold:\medskip

\noindent {\rm (1$\cc$)} For all $\bi_1\in \mc E_N$, 
\begin{align}
&\quad\;\E_{z_0}^{\bbeta,\bs w}\left[\frac{\e^{{A}^{\bbeta,\bs w}_0(t)} K_0^{\bbeta,\bs w}(0)}{K_0^{\bbeta,\bs w}(t)} f(\ms Z_t);
\left\{
\begin{array}{cc}
 t<T_0^{1}\\
\ms Z_{T_0^1}\in \CNwia
\end{array}
\right\}
\right]\notag\\
&=\E_{z_0}^{(0)}\left[\frac{w_{\bi_1} K_0^{\bbeta,\bi_1}(t)}{K_0^{\bbeta,\bs w}(t)} f(\ms Z_t)\E^{\bi_1}_{\ms Z_t}[\e^{-A^{\bbeta,\bs w,\bi_1}_0(T^{\bi_1}_0)}]\right].\label{FK:m=0}
\end{align}
Moreover,
\begin{align}\label{FK:m=0+}
\E^{\bbeta,\bs w}_{z_0}\left[\frac{\e^{{A}^{\bbeta,\bs w}_0( t)}K_0^{\bbeta,\bs w}(0)}{K_0^{\bbeta,\bs w}(t)}  f(\ms Z_t);t<T_0^1\right]=\E^{(0)}_{z_0}[f(\ms Z_t)].
\end{align}

\noindent {\rm (2$\cc$)} For all integers $m\geq 1$ and $\bi_1,\cdots,\bi_m\in \mc E_N$ with $\bi_1\neq \bi_2\neq \cdots\neq \bi_m$, it holds that
\begin{align}
&\quad\;\E^{\bbeta,\bs w}_{z_0}\left[\frac{\e^{{A}^{\bbeta,\bs w}_0(t)}K_0^{\bbeta,\bs w}(0)}{K_0^{\bbeta,\bs w}(t)}f(\ms Z_t);
\left\{
\begin{array}{cc}
T^m_0\leq t<T_0^{m+1}\\
\ms Z_{T_0^1}\in \CNwia,\cdots,\ms Z_{T_0^m}\in \CNwib
\end{array}
\right\}
\right]\notag\\
&=\int _{0< s_{1}< \cdots < s_{m}< t}P^{\bbeta ;{\mathbf i}_{1},\ldots ,{
\mathbf i}_{m}}_{s_{1},\ldots ,s_{m},t}f(z_{0})
\,\mathrm{d} \boldsymbol s_{m}.\label{sum:DBG}
\end{align}

\noindent {\rm (3$\cc$)} The Feynman--Kac-type formula in \eqref{goal:FK} holds.
\end{thm}

The proof of Theorem~\ref{thm:main2} will begin in Section~\ref{sec:FKproof} after we prove some preliminary stochastic integral representations in Section~\ref{sec:SGP2}. 

\begin{rmk}[Independence of weights]\label{rmk:weight}
The formulas in \eqref{goal:FK}, \eqref{FK:m=0+} and \eqref{sum:DBG} show that the expectations do not depend on the magnitudes of $w_\bi>0$. This ``cancellation'' effect is consistent with the fact that relatively larger $w_\bi$ means stronger interactions between $\{Z_t^{i\prime}\}$ and $\{Z^i_t\}$ via the SDE of $\{Z^\bi_t\}$ ($\bi=(i\prime,i)\in \mc E_N$), whereas larger $w_\bi$ reduces the integrals with respect to $\d L_s^\bi$ in \eqref{def:finalA} to ${A}^{\bbeta,\bs w}_0(\cdot)$. Here, for $t<T_\partial$, $Z_t^{\bi}$ takes values in $\Bbb C$ and obeys the following SDE:
\begin{align}\label{RM:SDE}
Z^\bi_t&=Z^\bi_0-\sum_{\bj\in \EN}\frac{\sigma(\bi)\cdot \sigma(\bj)}{2}\int_0^t \frac{w_\bj\hK_1^{\bbeta,\bj}(s)}{K_0^{\bbeta,\bs w}(s)}\left(\frac{1}{\overline{Z}^\bj_s}\right)\d s+\widetilde{W}^\bi_t,
\end{align}
where we set $\hK_1^{\bbeta,\bj}(s)\,\defeq\,\hK_1(\sqrt{2\beta_\bj}|Z^\bj_s|)$
using the Macdonald function $K_\nu(\cdot)$ of indices $\nu$ and $\widehat{K}_\nu(x)\,\defeq\,x^\nu K_\nu(x)$, and
$\{\wt{W}^\bi_t\}$ is a two-dimensional standard Brownian motion; see \cite[(3.10)]{C:SDBG2-3}. Also, $\sigma(\bk)\in \{-1,0,1\}^N$ denotes the column vector such that the $k\prime$-th component is $1$, the $k$-th component is $-1$, and the remaining components are zero when $\bk=(k\prime,k)\in \mc E_N$.
\qed  
\end{rmk}

\begin{rmk}[Normalization of the local times]\label{rmk:localtime}
The normalization of the local times $\{L^\bi_t\}$ in $\{A^{\bbeta,\bs w}_0(t)\}$
under $\P^{\bbeta,\bs w}$ can be obtained by reversing some steps of the proofs in this paper. We do not puruse the details here, but one can start by 
using the local absolute continuity between $\P^{\bbeta,\bs w}$ and the law of a stochastic one-$\delta$ motion over $[T_0^m,T_0^{m+1})$, to be recalled in Section~\ref{sec:FKproof}, and then referring to Section~\ref{sec:SGP2} for the normalization of local times under stochastic one-$\delta$ motions. \qed 
\end{rmk}

\addtocontents{toc}{\protect\setcounter{tocdepth}{1}}
\subsection{Representations by the stochastic one-$\bs \delta$ motions}\label{sec:SGP2}
\addtocontents{toc}{\protect\setcounter{tocdepth}{2}}
\addcontentsline{toc}{subsection}{\numberline{\thesubsection}{Representations by the stochastic one-$\delta$ motions}}
The first step toward the proof of Theorem~\ref{thm:main2} is Proposition~\ref{prop:main1} which we prove in this subsection. It gives some preliminary stochastic path integral representations of the summands in \eqref{P:series}. They are in terms of the stochastic one-$\delta$ motions $\{\ms Z_t\}=\{Z^j_t\}_{1\leq j\leq N}$ under 
$\P^\bi\,\defeq\,\P^{\beta_\bi\da,\bi}$
for $\bi\in \mc E_N$ \cite{C:SDBG1-3}: 
For all $N\geq 2$,  $z_0=(z_0^1,\cdots,z_0^N)\in \CN$, $\beta_\bi\in (0,\infty)$, $\bi=(i\prime,i)\in \mc E_N$, we set
\begin{align}\label{def:SDE2-3}
Z^j_t\,\defeq\, 
\begin{cases}
\displaystyle \frac{(z^{\bi\prime}_0+W^{\bi\prime}_t)+(\1_{j=i\prime}-\1_{j=i})Z^\bi_t}{\two},&j\in \bi,\\
\vspace{-.4cm}\\
z^j_0+W^j_t,&j\in \{1,\cdots,N\}\setminus \bi.
\end{cases}
\end{align}
Here, $\bi$ is also regarded as the set $\{i\prime,i\}$, $\mathcal E_N$ defined in \eqref{def:EN-3}, $z_0^{\bi\prime}$ has been defined in \eqref{def:unitary-3}, 
$\{Z^\bi_t\}$ is a version of the stochastic relative motion 
$\{Z_t\}$ under $\P^{\beta_\bi\da}_{z_0^\bi}$ from \cite{C:BES-3}, and  $\{W^{\bi\prime}_t\}\cup \{W^k_t\}_{k\in \{1,\cdots,N\}\setminus\bi}$ consist of 
$N-1$ many independent two-dimensional standard Brownian motions with zero initial conditions and independent of $\{Z^\bi_t\}$. Also, under $\P^\bi$, $\{Z^\bi_t\}$ admits a Markovian local time $\{L^\bi_t\}$, which is chosen to be subject to the normalization by Donati-Martin and Yor~\cite[Corollary~2.3]{DY:Krein-3} when viewed as the local time of $\{|Z^\bi_t|\}$ at level $0$; recall Remark~\ref{rmk:localtime}.
 The following lemma restates part of \cite[Theorem~2.1]{C:SDBG1-3}.

\begin{lem}
{\rm (1$\cc$)} Let $z^0\in\Bbb C$. For all $h\in \B_+(\R_+)$,
\begin{align}\label{DY:law2-3}
&\E^{\beta\da}_{z^0}\left[\int_0^t h(\tau)\d L_\tau \right]=
\begin{cases}
\displaystyle \int_0^t
\frac{P_{2s}(\two z^0)}{2K_0(\sqrt{2\beta}|z^0|)}
\int_{s}^{t}  \e^{-\beta \tau}\s^{\beta}(\tau-s)h(\tau)\d\tau \d s,&z^0\neq 0,\\
\vspace{-.4cm}\\
\displaystyle \int_0^t \frac{\e^{-\beta \tau}\s^\beta(\tau)}{4\pi}h(\tau)\d \tau,&z^0=0.
\end{cases}
\end{align}
\noindent {\rm (2$\cc$)} Let $z_0\in \CN$. For any $0<t<\infty$ and $F\in \B_+(\CN)$,
\begin{align}
\begin{split}
\label{DY:com-3}
\1_{\{t<T_0^\bi\}}\d \P_{z_0}^{\bi}&=\frac{\e^{-\beta_\bi t}K_0(\sqrt{2\beta_\bi}|Z^\bi_t|)}{K_0(\sqrt{2\beta_\bi}|Z^\bi_0|)}\d \P_{z_0}^{(0)}\;\;\mbox{ on }\F_t^0,\;\forall\; z_0:z_0^\bi\neq 0,
\end{split}\\
\E^{\bi}_{z_0}\left[\frac{\e^{\beta_\bi t}F(\ms Z_t)}{2K_0(\sqrt{2\beta_\bi}|Z^\bi_t|)};T^\bi_0\leq t\right]&=\E^{\bi}_{z_0}\left[\int_0^t\e^{\beta_\bi \tau}\E_{\ms Z_{\tau}}^{(0)}[F(\ms Z_{t-\tau})]
 \d L^\bi_\tau \right],\; \forall\; z_0\in \CN,\label{exp:LT-3}
\end{align}
where $\F_t^0\,\defeq\,\sigma(\ms Z_s;s\leq t)$, $\{\ms Z_t\}$ under $\P_{z_0}^{(0)}$ is a $2N$-dimensional standard Brownian motion starting from $z_0$, and $T^\bi_0\,\defeq\inf\{t\geq 0;Z^\bi_t=0\}$.
\end{lem}

In the sequel, we also work with $\P^\bi_{z_1,s}$ for $z_1\in \CN$ and $s\geq 0$ defined as follows: 
\begin{align}\label{shiftprob}
\P^{\bi,s}_{z_1}(\{\ms Z_{t};t\geq s\}\in \Gamma)\,\defeq\,\P^{\bi}_{z_1}(\{\ms Z_t;t\geq 0\}\in \Gamma). 
\end{align}
Hence, under $\P^{\bi}_{z_1,s}$, the stochastic one-$\delta$ motion starts at time $s$. Accordingly, these probability measures are more convenient for ``undoing'' the Markov property, which is a basic idea in the forthcoming proofs. 

\begin{prop}\label{prop:main1}
Fix $\bbeta\in (0,\infty)^{\mc E_N}$, $z_0\in \CN_{\sp}$, $0<t<\infty$ and $0\leq f\in \B_b(\CN)$.
For all $m\in \Bbb N$ and $\bi_1,\cdots,\bi_m\in \mc E_N$ with $\bi_1\neq \bi_2\neq \cdots\neq \bi_m$, it holds that 
\begin{align}
&\quad\;\int _{0< s_{1}< \cdots < s_{m}< t}\mathrm{d} \boldsymbol s_{m} P^{\bbeta;{\mathbf i}_{1},\ldots ,{
\mathbf i}_{m}}_{s_{1},\ldots ,s_{m},t}f(z_{0})\notag\\
&=2^mK_0^{\bbeta,\bi_1}(0)\int_\Omega \d \P^{\bi_1}_{z_0}(\omega_1)\int_0^t  \d L^{\bi_1}_{\tau_1}(\omega_1)\e^{\beta_{\bi_1} \tau_1}\notag\\
&\quad\times K^{\bbeta,\bi_{2}}_0(\tau_1)(\omega_1)\int_\Omega \d \P^{\bi_2, \tau_1}_{ \ms Z_{\tau_1}(\omega_1)}(\omega_2)\int_{\tau_1}^{t}  \d L^{\bi_2}_{\tau_2}(\omega_2)\e^{\beta_{\bi_2} (\tau_2-\tau_1)}\times \cdots\notag\\
&\quad  \times  K^{\bbeta,\bi_{m}}_0(\tau_{m-1})(\omega_{m-1})\int_\Omega \d \P^{\bi_m,\tau_{m-1}}_{ \ms Z_{\tau_{m-1}}(\omega_{m-1})}(\omega_m)\int_{\tau_{m-1}}^{t}  \d L^{\bi_m}_{\tau_m}(\omega_m)\e^{\beta_{\bi_m} (\tau_m-\tau_{m-1})}\notag\\
&\quad\times \E_{\ms Z_{\tau_m}(\omega_m)}^{(0)}[f(\ms Z_{t-\tau_m})].\label{rep:Plimit}
\end{align}
\end{prop}

The following lemma handles the key mechanism for the proof of Proposition~\ref{prop:main1}.

\begin{lem}\label{lem:conn}
For all $\bi\in \mc E_N$, $z_0\in \CN$ with $z_0^\bi\neq 0$, $F\in \B_+( \CN\times [0,t])$ and $0\leq \tau'<t<\infty$,
\begin{align}
&\quad\;2K_0(\sqrt{2\beta_\bi}|z^\bi_0|)\E^{\bi,\tau'}_{z_0}\left[\int_{\tau'}^t (\d L^\bi_\tau) \e^{\beta_\bi(\tau-\tau')}F(\ms Z_\tau,\tau)\right]\notag\\
&=\int_{\tau'}^t\d s   
\int_{s}^t \d \tau 
\int _{{\mathbb C}^{N-1}}\mathrm{d}z^{/{\mathbf i}}K_{\tau',s}(z_0,z^{/\bi})
 \int_{\Bbb C^{N-1}}
\d \tilde{z}^{\backslash\bi} 
K^{\beta_\bi;\bi;\interact}_{s,\tau}(z^{/\bi},\tilde{z}^{\backslash\bi})F(\tilde{z}^{\backslash\!\backslash\bi},\label{formula:conn}
\tau),
\end{align}
where the variables for integration on the right-hand side are defined in Theorem~\ref{thm:DBGsgp}. 
\end{lem}
\begin{proof}
The main step of the proof is to find the explicit formula of the expectation in \eqref{formula:conn}. To this end,
we first shift time back to $0$ from $\tau'$ so that
the expectation in \eqref{formula:conn} satisfies 
\begin{align}
\E^{\bi,\tau'}_{z_0}\biggl[\int_{\tau'}^t (\d_\tau L^\bi_\tau) \e^{\beta_\bi(\tau-\tau')}F(\ms Z_\tau,\tau)\biggr]
&=\E^{\bi}_{z_0}\biggl[\int_{0}^{t-\tau'} (\d_\tau L^\bi_\tau)\e^{\beta_\bi \tau}F(\ms Z_\tau,\tau'+\tau)\biggr]\notag\\
&=\E^{\bi}_{z_0}\biggl[\int_{0}^{t-\tau'} (\d_\tau L^\bi_\tau) \e^{\beta_\bi \tau}
F( \ms Z_\tau^{\backslash\!\backslash\!\backslash \bi}
,\tau'+\tau)\biggr].\label{conn:1}
\end{align}
Specifically, the first equality follows by replacing $\tau-\tau'$ on its left-hand side with $\tau$. Also, in the second equality, $\ms Z_\tau^{\backslash\!\backslash\!\backslash\bi}$ is a random vector such that the $k$-th components are given by 
\begin{align*}
Z_\tau^{k\bbb\bi}\,\defeq\begin{cases}
(Z_0^{\bi\prime}+W_\tau^{\bi\prime})/\two ,&k\in \bi,\\
Z^k_0+W^k_\tau,&k\notin \bi,
\end{cases}
\end{align*}
and \eqref{conn:1} holds
since $\d_\tau L_\tau^\bi$ is supported in $\{\tau;Z^\bi_\tau=0\}$, and so, by \eqref{def:SDE2-3}, 
\begin{align}\label{eq:Ztbstate}
\mbox{under }\P^{\bi},\quad
Z^\bi_\tau=0\Longrightarrow
\begin{cases}
\displaystyle Z^{i\prime}_\tau=\frac{Z^{\bi\prime}_\tau+Z^{\bi}_\tau}{\two}=\frac{Z^{\bi\prime}_\tau}{\two}=\frac{Z_0^{\bi\prime}+W_\tau^{\bi\prime}}{\two},\\
\vspace{-.4cm}\\
\displaystyle Z^{i}_\tau=\frac{Z^{\bi\prime}_\tau-Z^{\bi}_\tau}{\two}=\frac{Z^{\bi\prime}_\tau}{\two}=\frac{Z_0^{\bi\prime}+W_\tau^{\bi\prime}}{\two},\\
\vspace{-.4cm}\\
Z^k_\tau=Z^k_0+W^k_\tau, \quad\forall\; k\notin \bi.
\end{cases}
\end{align}

Next, we continue from \eqref{conn:1} by using the notation in \eqref{def:uslash}--\eqref{def:udbslash}, the notation of multiplication columns defined in \eqref{def:column-3}, and
the independence of the $N$ processes $\{W_t^{\bi\prime}\}$,  $\{Z^\bi_t\}$ and $\{W_t^k\}$ for $k\notin \bi$. Also, let $\tilde{z}^{\bi\prime}$ denote the state of $W^{\bi\prime}_\tau$.
This gives
\begin{align}
&\quad\;\E^{\bi,\tau'}_{z_0}\left[\int_{\tau'}^t (\d_\tau L^\bi_\tau) \e^{\beta_\bi(\tau-\tau')}F(\ms Z_\tau,\tau)\right]\notag\\
&=\E^{\bi}_{z_0}\biggl[\int_{0}^{t-\tau'} (\d_\tau L^\bi_\tau) \e^{\beta_\bi \tau}\int_{\Bbb C^{N-1}}\d \tilde{z}^{\backslash\bi}
\begin{bmatrix}
P_{\tau}(z_0^{\bi\prime},\tilde{z}^{\bi\prime}
)\\
\prod_{k:k\notin \bi}P_{\tau}(z_0^k,\tilde{z}^k)
\end{bmatrix}_\times 
F( \tilde{z}^{\backslash\!\backslash \bi},\tau'+\tau)\biggr]\notag,
\end{align}
where $\tilde{z}^{\backslash\!\backslash \bi}$ is a vector of $\R^2$-variables 
defined in \eqref{def:udbslash}
 and $\d \tilde{z}^{\backslash \bi}$ is defined in \eqref{def:integrators}. Moreover, the last integral can be made fully explicit as an iterated Riemann integral by using \eqref{DY:law2-3} for $z^0=z_0^\bi\neq 0$, $t$ replaced by $t-\tau'$, and $\beta=\beta_\bi$:
\begin{align}
&\quad\;\E^{\bi,\tau'}_{z_0}\left[\int_{\tau'}^t (\d_\tau L^\bi_\tau) \e^{\beta_\bi(\tau-\tau')}F(\ms Z_\tau,\tau)\right]\notag\\
&=\int_0^{t-\tau'}\d s
\frac{P_{2s}(\two z^\bi_0)}{2K_0(\sqrt{2\beta_\bi}|z^\bi_0|)}
\int_{s}^{t-\tau'}\d \tau  \e^{-\beta_\bi \tau}\s^{\beta_\bi}(\tau-s)\e^{\beta_\bi \tau}\int_{\Bbb C^{N-1}}\d \tilde{z}^{\backslash \bi}
\begin{bmatrix}
P_{\tau}(z_0^{\bi\prime},\tilde{z}^{\bi\prime}
)\\
\prod_{k:k\notin \bi}P_{\tau}(z_0^k,\tilde{z}^k)
\end{bmatrix}_\times \notag\\
&\quad\;\times F(\tilde{z}^{\backslash\!\backslash \bi},\tau'+\tau)\notag\\
&=\int_{\tau'}^{t}\d s'
\frac{P_{2(s'-\tau')}(\two z^\bi_0)}{2K_0(\sqrt{2\beta_\bi}|z^\bi_0|)}
\int_{s'}^{t}\d \tau''  \s^{\beta_\bi}(\tau''-s')\int_{\Bbb C^{N-1}}\d \tilde{z}^{\backslash \bi}
\begin{bmatrix}
P_{\tau''-\tau'}(z_0^{\bi\prime},\tilde{z}^{\bi\prime}
)\\
\prod_{k:k\notin \bi}P_{\tau''-\tau'}(z_0^k,\tilde{z}^k)
\end{bmatrix}_\times\notag \\
&\quad \times F( \tilde{z}^{\backslash\!\backslash \bi},\tau'')\label{ChapKol:00}
\end{align}
by changing variables according to $s=s'-\tau'$ and $\tau=\tau''-\tau'$. 
 To simplify the right-hand side, we use the Chapman--Kolmogorov equation to get \eqref{ChapKol:0} and the equality next to it: for $\tau'<s'<\tau''$,
\begin{align}
\prod_{k:k\notin \bi}P_{\tau''-\tau'}(z_0^k,\tilde{z}^k)&=\int_{\Bbb C^{N-2}}\bigotimes_{k:k\notin \bi}\d z^{k}  \biggl(\prod_{k:k\notin \bi}P_{s'-\tau'}(z_0^k,z^k)\biggr)\notag\\
&\quad \;\biggl(\prod_{k:k\notin \bi}P_{\tau''-s'}(z^k,\tilde{z}^k)\biggr),\label{ChapKol:0}\\
P_{2(s'-\tau')}(\two z_0^{\bi})P_{\tau''-\tau'}(z_0^{\bi\prime},\tilde{z}^{\bi\prime})&=P_{2(s'-\tau')}(\two z_0^{\bi},0)\int _{\Bbb C}\d z^i2P_{s'-\tau'}(z_0^{\bi\prime},\two z^i)P_{\tau''-s'}(\two z^i,\tilde{z}^{\bi\prime})\notag\\
&=\int_{\Bbb C}\d z^iP_{s'-\tau'}(z_0^{\bi\prime},\two z^i)P_{s'-\tau'}(z_0^{\bi},0)P_{\tau''-s'}(\two z^i, \tilde{z}^{\bi\prime})\notag\\
&=\int_{\Bbb C}\d z^iP_{s'-\tau'}(z_0^{i\prime},z^i)P_{s'-\tau'}(z_0^i,z^i)P_{\tau''-s'}(\two z^i,\tilde{z}^{\bi\prime}),\label{ChapKol:1}
\end{align}
where the last equality uses the following readily verified identity:
\[
P_t(z^0,z^1)P_t(\tilde{z}^0,\tilde{z}^1)=P_t\left(\frac{z^0+\tilde{z}^0}{\two},\frac{z^1+\tilde{z}^1}{\two}\right)P_t\left(\frac{z^0-\tilde{z}^0}{\two},\frac{z^1-\tilde{z}_1}{\two}\right).
\]
Applying \eqref{ChapKol:0} and \eqref{ChapKol:1} to \eqref{ChapKol:00} gives
\begin{align}
&\quad\;\E^{\bi,\tau'}_{z_0}\left[\int_{\tau'}^t (\d_\tau L^\bi_\tau) \e^{\beta_\bi(\tau-\tau')}F(\ms Z_\tau,\tau)\right]\notag\\
\begin{split}
&=\frac{1}{2K_0(\sqrt{2\beta_\bi}|z^\bi_0|)}\int_{\tau'}^t\d s'   
\int_{s'}^t \d \tau'' 
\int _{{\mathbb C}^{N-1}}\mathrm{d}z^{/{\mathbf i}}\int_{\Bbb C^{N-1}}
\d \tilde{z}^{\backslash\bi}
\\
&\quad\; \begin{bmatrix}
P_{s'-\tau'}(z_0^{i\prime},z^{i})\\
P_{s'-\tau'}(z_0^{i},z^{i})\\
\prod_{k:k\notin \bi}P_{s'-\tau'}(z_0^{k},z^{k})
\end{bmatrix}_\times 
\begin{bmatrix}
P_{\tau''-s'}(\two z^{i},\tilde{z}^{\bi\prime}
)\\
\s^{\beta_\bi}(\tau''-s')\\
\prod_{k:k\notin \bi}P_{\tau''-s'}(z^{k},\tilde{z}^k)
\end{bmatrix}_\times  F(\tilde{z}^{\backslash\!\backslash \bi},\tau''),\label{conn:2}
\end{split}
\end{align}
where we use the integrator $\mathrm{d}z^{/{\mathbf i}}$ defined in \eqref{def:integrators}.
To finish the proof, note that $z^i,z^k$ in \eqref{conn:2} can be rewritten as $z^{i/\bi}=z^{i\prime/\bi},z^{k/\bi}$, respectively, by \eqref{def:uslash},
and recall  the kernels defined in \eqref{def:Kbeta} and \eqref{def:Kbetainteract}. So, \eqref{formula:conn} follows upon multiplying both sides of \eqref{conn:2} by $2K_0(\sqrt{2\beta_\bi}|z^\bi_0|)$. 
\end{proof}

\begin{proof}[Proof of Proposition~\ref{prop:main1}]
Fix $t>0$.
Let us begin by recalling the definition \eqref{def:Qsummands} of  $P^{\bbeta;{\mathbf i}_{1},\ldots ,{
\mathbf i}_{m}}_{s_{1},\ldots ,s_{m},t}f(z_{0})$ and
rewriting the right-hand side of \eqref{def:Qsummands}
such that 
 the kernels $ K^{\beta ;{\mathbf i}}_{s,\tau ,s'}$  and $K^{\beta;\bi;\interact}_{s,\tau}$ defined in  \eqref{def:Kbeta}  and \eqref{def:Kbetainteract} are explicitly displayed:
 \begin{align*}
 &\quad\;\int _{0< s_{1}< \cdots < s_{m}< t}\mathrm{d} \boldsymbol s_{m} P^{\bbeta;{\mathbf i}_{1},\ldots ,{
\mathbf i}_{m}}_{s_{1},\ldots ,s_{m},t}f(z_{0})\notag\\
&=\int_{0<s_1<\cdots<s_m<t}\d \bs s_m \int_{\Bbb C^{N-1}}\d z_1^{/\bi_1}K_{0,s_1}(z_0,z_1^{/\bi_1})\\
&\quad\;\int_{\Bbb C^{N-1}}\d z_2^{/\bi_2}
\int_{s_1}^{s_2}\d \tau_1\int_{\Bbb C^{N-1}}\d \tilde{z}_1^{\backslash\bi_1} K^{\beta_{\bi_1};\bi_1;\interact}_{s_1,\tau_1}(z_1^{/\bi_1},\tilde{z}_1^{\backslash\bi_1})K_{\tau_1,s_2}(\tilde{z}^{\backslash\!\backslash\bi_1}_1,z_2^{/\bi_2})\cdots \\
&\quad\;\int_{\Bbb C^{N-1}}\d z_m^{/\bi_m}
\int_{s_{m-1}}^{s_m}\d \tau_{m-1}\int_{\Bbb C^{N-1}}\d \tilde{z}_{m-1}^{\backslash\bi_{m-1}} K^{\beta_{\bi_{m-1}};\bi_{m-1};\interact}_{s_{m-1},\tau_{m-1}}(z_{m-1}^{/\bi_{m-1}},\tilde{z}_{m-1}^{\backslash\bi_{m-1}})\\
&\quad \times K_{\tau_{m-1},s_m}(\tilde{z}^{\backslash\!\backslash\bi_{m-1}}_{m-1},z_m^{/\bi_m})\\
&\quad\; \int_{\CN}\d z_{m+1}\int_{s_m}^t \d \tau_m \int_{\Bbb C^{N-1}}\d \tilde{z}_m^{\backslash \bi_m}K^{\beta_{\bi_m};\bi_m;\interact}_{s_m,\tau_m}(z_m^{/\bi_m},\tilde{z}_m^{\backslash \bi_m})K_{\tau_m,t}(\tilde{z}_m^{\backslash\!\backslash \bi_m},z_{m+1})f(z_{m+1}).
 \end{align*}
 To use this expression, first, we change the order of time integration as follows:
 \begin{align*}
 \int_{0<s_1<\cdots<s_m<t}\d \bs s_m\int_{s_1}^{s_2}\d \tau_1\cdots \int_{s_m}^t \d \tau_m &=\int_{0<s_1<\tau_1<s_2<\cdots<s_m<\tau_m<t}\\
 &=\int_0^t \d s_1\int_{s_1}^t \d \tau_1\int_{\tau_1}^t\d s_2\cdots \int_{\tau_{m-1}}^t\d s_m\int_{s_m}^t \d \tau_m. 
 \end{align*}
Along with a change of the order of space integration, we get
 \begin{align}
 &\quad\;\int _{0< s_{1}< \cdots < s_{m}< t}\mathrm{d} \boldsymbol s_{m} P^{\bbeta;{\mathbf i}_{1},\ldots ,{
\mathbf i}_{m}}_{s_{1},\ldots ,s_{m},t}f(z_{0})\notag\\
&=\int_{0}^t \d s_1\int_{s_1}^t \d \tau_1 \int_{\Bbb C^{N-1}}\d z_1^{/\bi_1}K_{0,s_1}(z_0,z_1^{/\bi_1}) \int_{\Bbb C^{N-1}}\d \tilde{z}_1^{\backslash \bi_1} K^{\beta_{\bi_1};\bi_1;\interact}_{s_1,\tau_1}(z_1^{/\bi_1},\tilde{z}_1^{\backslash \bi_1})\notag\\
&\quad\;\int_{\tau_1}^t \d s_2\int_{s_2}^t \d \tau_2 \int_{\Bbb C^{N-1}}\d z_2^{/\bi_2}K_{\tau_1,s_2}(\tilde{z}^{\backslash\!\backslash\bi_{1}}_{1},z_2^{/\bi_2}) \int_{\Bbb C^{N-1}}\d \tilde{z}_2^{\backslash \bi_2} K^{\beta_{\bi_2};\bi_2;\interact}_{s_2,\tau_2}(z_2^{/\bi_2},\tilde{z}_2^{\backslash \bi_2})\cdots\notag\\
&\quad\;\int_{\tau_{m-2}}^t \d s_{m-1}\int_{s_{m-1}}^t \d \tau_{m-1} \int_{\Bbb C^{N-1}}\d z_{m-1}^{/\bi_{m-1}}K_{\tau_{m-2},s_{m-1}}(\tilde{z}^{\backslash\!\backslash\bi_{m-2}}_{m-2},z_{m-1}^{/\bi_{m-1}})\notag\\
&\quad\; \int_{\Bbb C^{N-1}}\d \tilde{z}_{m-1}^{\backslash \bi_{m-1}} K^{\beta_{\bi_{m-1}};\bi_{m-1};\interact}_{s_{m-1},\tau_{m-1}}(z_{m-1}^{/\bi_{m-1}},\tilde{z}_{m-1}^{\backslash \bi_{m-1}})\notag\\
&\quad\;\int_{\tau_{m-1}}^t \d s_m\int_{s_m}^t \d \tau_m \int_{\Bbb C^{N-1}}\d z_m^{/\bi_m}K_{\tau_{m-1},s_m}(\tilde{z}^{\backslash\!\backslash\bi_{m-1}}_{m-1},z_m^{/\bi_m})\notag \\
&\quad\;\int_{\Bbb C^{N-1}}\d \tilde{z}_m^{\backslash \bi_m} K^{\beta_{\bi_m};\bi_m;\interact}_{s_m,\tau_m}(z_m^{/\bi_m},\tilde{z}_m^{\backslash \bi_m})\notag\\
&\quad\;\int_{\CN}\d z_{m+1}K_{\tau_m,t}(\tilde{z}_m^{\backslash\!\backslash \bi_m},z_{m+1})f(z_{m+1}).\label{propmain1:1}
\end{align}
In the following three steps, we show the required probabilistic representation by viewing the iterated integral on the right-hand side of \eqref{propmain1:1}
 backward.\medskip 

\noindent {\bf Step 1.}
By the definition \eqref{def:K} of $K_{\tau,s'}$, the integral in the last line of \eqref{propmain1:1} satisfies 
\begin{align}
\int_{\CN}\d z_{m+1} K_{\tau_m,t}(\tilde{z}_m^{\backslash\!\backslash \bi_m},z_{m+1})f(z_{m+1})=\E^{(0)}_{\tilde{z}_m^{\backslash\!\backslash \bi_m}}[f(\ms Z_{t-\tau_m})]\label{propmain1:2}
\end{align}
since, by assumption, the components $\{Z^k_t\}$ of $\{\ms Z_t\}$ under $\P^{(0)}$ are independent two-dimensional standard Brownian motions.\medskip 

\noindent {\bf Step 2.} Next, we handle the remaining iterated integral of a smaller fold in \eqref{propmain1:1}.
For all $2\leq \ell\leq m$ and nonnegative $F$, it follows from Lemma~\ref{lem:conn} that
\begin{align}
&\quad\;\int_{\tau_{\ell-1}}^t \d s_{\ell}\int_{s_{\ell}}^t \d \tau_{\ell} \int_{\Bbb C^{N-1}}\d z_\ell^{/\bi_\ell}K_{\tau_{\ell-1},s_\ell}(\tilde{z}^{\backslash\!\backslash\bi_{\ell-1}}_{\ell-1},z_\ell^{/\bi_\ell}) \int_{\Bbb C^{N-1}}\d \tilde{z}_\ell^{\backslash \bi_\ell} K^{\beta_{\bi_\ell};\bi_\ell;\interact}_{s_\ell,\tau_{\ell}}(z_\ell^{/\bi_\ell},\tilde{z}_\ell^{\backslash \bi_\ell})F(z_\ell^{\backslash\!\backslash \bi_\ell},\tau_\ell)\notag\\
&=2K_0(\sqrt{2\beta_{\bi_\ell}}|(\tilde{z}^{\backslash\!\backslash\bi_{\ell-1}}_{\ell-1})^{\bi_\ell}|)\E^{\bi_\ell,\tau_{\ell-1}}_{\tilde{z}^{\backslash\!\backslash\bi_{\ell-1}}_{\ell-1}}\biggr[\int_{\tau_{\ell-1}}^t (\d_{\tau_\ell} L^{\bi_\ell}_{\tau_\ell})\e^{\beta_{\bi_\ell}(\tau_\ell-\tau_{\ell-1})}F(\ms Z_{\tau_\ell},\tau_\ell)\biggr].\label{propmain1:3}
\end{align}
Note that the right-hand side depends on the space and time variables only through $\tilde{z}^{\backslash\!\backslash\bi_{\ell-1}}_{\ell-1}$ and $\tau_{\ell-1}$, as we have kept $t$ fixed. Similarly, we have
\begin{align}
&\quad\;\int_{0}^t \d s_1\int_{s_1}^t \d \tau_1 \int_{\Bbb C^{N-1}}\d z_1^{/\bi_1}K_{0,s_1}(z_0,z_1^{/\bi_1}) \int_{\Bbb C^{N-1}}\d \tilde{z}_1^{\backslash \bi_1} K^{\beta_{\bi_1};\bi_1;\interact}_{s_1,\tau_1}(z_1^{/\bi_1},\tilde{z}_1^{\backslash \bi_1})F(z_1^{\backslash\!\backslash \bi_1},\tau_1)\notag\\
&=2K_0(\sqrt{2\beta_{\bi_1}}|z_0^{\bi_1}|)\E^{\bi_1}_{z_0}\left[\int_{0}^t (\d_{\tau_1} L^{\bi_1}_{\tau_1}) \e^{\beta_{\bi_1} \tau_1}F(\ms Z_{\tau_1},\tau_1)\right],\label{propmain1:4}
\end{align}
since $\E^{\bi}_{z_0,0}=\E^{\bi}_{z_0}$ by definition.\medskip 

\noindent {\bf Step 3.} Finally, apply \eqref{propmain1:2}--\eqref{propmain1:4} to the right-hand side of \eqref{propmain1:1}. This way, we find that
\begin{align}
 &\quad\;\int _{0< s_{1}< \cdots < s_{m}< t}\mathrm{d} \boldsymbol s_{m} P^{\bbeta;{\mathbf i}_{1},\ldots ,{
\mathbf i}_{m}}_{s_{1},\ldots ,s_{m},t}f(z_{0})\notag\\
&=2K_0(\sqrt{2\beta_{\bi_1}}|z^{\bi_1}_{0}|)\int_\Omega \d \P^{\bi_1}_{z_0}(\omega_1)\int_0^t  \d  L^{\bi_1}_{\tau_1}(\omega_1)\e^{\beta_{\bi_1} \tau_1}\notag\\
&\quad\times 2K_0(\sqrt{2\beta_{\bi_2}}|Z^{\bi_2}_{\tau_1}(\omega_1)|)\int_\Omega \d \P^{\bi_2,\tau_1}_{ \ms Z_{\tau_1}(\omega_1)}(\omega_2)\int_{\tau_1}^{t}  \d L^{\bi_2}_{\tau_2}(\omega_2)\e^{\beta_{\bi_2} (\tau_2-\tau_1)}\times \cdots\notag\\
&\quad  \times  2K_0(\sqrt{2\beta_{\bi_m}}|Z^{\bi_m}_{\tau_{m-1}}(\omega_{m-1})|)\int_\Omega \d \P^{\bi_m,\tau_{m-1}}_{ \ms Z_{\tau_{m-1}}(\omega_{m-1})}(\omega_m)\int_{\tau_{m-1}}^{t}  \d L^{\bi_m}_{\tau_m}(\omega_m)\e^{\beta_{\bi_m} (\tau_m-\tau_{m-1})}\notag\\
&\quad\times \E_{\ms Z_{\tau_m}(\omega_m)}^{(0)}[f(\ms Z_{t-\tau_m})],\notag
\end{align}
which is the required formula in \eqref{rep:Plimit} upon writing 
\[
{\textstyle 2K_0(\sqrt{2\beta_{\bi_1}}|z^{\bi_1}_{0}|)=2K_0^{\bbeta,\bi_1}(0)},\quad 
{\textstyle 2K_0(\sqrt{2\beta_{\bi_{\ell+1}}}|Z^{\bi_{\ell+1}}_{\tau_\ell}(\omega_\ell)|)}=2K^{\bbeta,\bi_{\ell+1}}_0(\tau_\ell)(\omega_\ell),\quad 1\leq \ell\leq m.
\]
The proof is complete.
\end{proof}

\subsection{Proof of the Feynman--Kac-type formulas}\label{sec:FKproof}
To prove Theorem~\ref{thm:main2}, we need a few more notations and identities for the stochastic many-$\delta$ motions. These additional notations extend those stated before Theorem~\ref{thm:main2}. 

First, for ${A}^{\bbeta,\bs w}_0(t)$ defined in \eqref{def:finalA}, we write
\begin{align}
{A}^{\bbeta,\bs w}_0(t)&= \mathring{A}^{\bbeta,\bs w}_0(t)+\overline{A}^{\bbeta,\bs w}_0(t),\label{dec:finalA}
\end{align}
where
\begin{align}
\mathring{A}^{\bbeta,\bs w}_0(t)&\,\defeq\, \sum_{\bi\in\mc E_N}\mathring{A}^{\bbeta,\bs w,\bi}_0(t),\label{def:finalrA}\\
\mathring{A}^{\bbeta,\bs w,\bi}_0(t)&\,\defeq\, \sum_{\bj\in\mathcal E_N\setminus\{ \bi\}}2\left(\frac{w_\bj}{w_\bi}\right)\int_0^t K_0(\sqrt{2\beta_\bj} |Z^\bj_s|)\d L^{\bi}_s,\label{def:Aring-3}\\
\overline{A}^{\bbeta,\bs w}_0(t)&\,\defeq\, \sum_{\bj\in \mathcal E_N} \int_0^t  \beta_\bj\frac{w_\bj K^{\bbeta,\bj}_0(s)}{K^{\bbeta,\bs w}_0(s)}\d s.\label{def:finalhA}
\end{align}
Next, the following identities are taken from \cite[Propositions~3.11 and~3.15]{C:SDBG2-3}:
\begin{align}
 &\E^{\bbeta,\bs w}_{z_0}[F(\ms Z_{t\wedge T_0^{\mc E_N}};t\geq 0)\e^{A^{\bbeta,\bs w,\bi}_0(T^{\mc E_N}_0)};T^{\mc E_N}_0=T^\bi_0]=\frac{w_\bi K_0^{\bbeta,\bi}(0)}{K_0^{\bbeta,\bs w}(0)}\E^{\bi}_{z_0}[F(\ms Z_{t\wedge T_0^{\bi}};t\geq 0)], \label{Pinhomo:id-3}\\
&\sum_{\bi\in \mc E_N}\frac{w_\bi K_0^{\bbeta,\bi}(0)}{K_0^{\bbeta,\bs w}(0)}\E^{\bi}_{z_0}[\e^{-A^{\bbeta,\bs w,\bi}_0(T^{\mc E_N}_0)}]=1,\label{Pinhomo:idsum-3}\\
&\E^{\bbeta,\bs w,\bi}_{z_0}\left[F(\ms Z_{t\wedge t_0};t\geq 0);t_0<\Twi_0\right] =\E^{\bi}_{z_0}\left[F(\ms Z_{t\wedge t_0};t\geq 0)\mathcal E^{\bbeta,\bs w,\bi}_{z_0}(t_0)
\right]
,\label{def:Q21-3}
\end{align} 
where $z_0\in \CNw$ is imposed for \eqref{Pinhomo:id-3} and \eqref{Pinhomo:idsum-3}, $z_0\in \CNwni\,\defeq \CNw\cup\; \CNwi$ is imposed for \eqref{def:Q21-3}, $T_0^\bi$ is defined below \eqref{exp:LT-3},
$F\in \B_+(C_{\Bbb C^N}[0,\infty))$, $0<t_0<\infty$, and
\begin{align}
A^{\bbeta,\bs w,\bi}_0(t)&\,\defeq\, \mathring{A}^{\bbeta,\bs w,\bi}_0(t)+\overline{A}^{\bbeta,\bs w}_0(t)-\beta_\bi t,
\label{def:A-3}\\
\mathcal E_{z_0}^{\bbeta,\bs w,\bi}(t)&\,\defeq\,\frac{w_\bi K_0^{\bbeta,\bi}(0)}{K_0^{\bbeta,\bs w}(0)}
\frac{\e^{-A^{\bbeta,\bs w,\bi}_0(t)}K_0^{\bbeta,\bs w}(t)}{w_\bi K_0^{\bbeta,\bi}(t)}.\label{superMG-3}
\end{align}
Here in \eqref{def:Q21-3} and what follows,
we write $\P^{\bbeta,\bs w,\bi}_{z_0}$ for $\P^{\bbeta,\bs w}_{z_0}$ to stress the dependence on $\bi$ when $z_0\in \CNwni$. Also, $w_\bi K_0^{\bbeta,\bi}(s)/K_0^{\bbeta,\bs w}(s)$, $s\in \{0,t\}$, in $\mathcal E^{\bbeta,\bs w,\bi}_{z_0}(t)$ in \eqref{superMG-3} is defined to be $1$ when $Z_s^{\bi}=0$ (cf. \cite[Proposition~3.7 (1$\cc$)]{C:SDBG2-3}).
Note that \eqref{Pinhomo:id-3}, \eqref{Pinhomo:idsum-3} and \eqref{def:Q21-3} all show explicit relationships between the stochastic many-$\delta$ motion under $\P^{\bbeta,\bs w}$ and the stochastic one-$\delta$ motions under $\P^{\bi}$. These relationships are our tools to use the preliminary stochastic path integral representations in Proposition~\ref{prop:main1}.
\medskip

\begin{proof}[Proof of Theorem~\ref{thm:main2} (1$\cc$)]
By \eqref{Pinhomo:id-3}, we have
\begin{align}
&\quad\;\E_{z_0}^{\bbeta,\bs w}\left[\frac{\e^{A^{\bbeta,\bs w,\bi_1}_0(t)+\beta_{\bi_1}t}K_0^{\bbeta,\bs w}(0)}{K_0^{\bbeta,\bs w}(t)} f(\ms Z_t);
\left\{
\begin{array}{cc}
 t<T_0^{1}\\
\ms Z_{T_0^1}\in \CNwia
\end{array}
\right\}
\right]\notag\\
&=\frac{w_{\bi_1} K_0^{\bbeta,\bi_1}(0)}{K^{\bbeta,\bs w}_0(0)}\E_{z_0}^{\bi_1}\left[\e^{-A_0^{\bbeta,\bs w,\bi_1}(T_0^{\bi_1})}\frac{\e^{A^{\bbeta,\bs w,\bi_1}_0(t)+\beta_{\bi_1}t}K_0^{\bbeta,\bs w}(0)}{K_0^{\bbeta,\bs w}(t)} f(\ms Z_t);
t<T_0^{\bi_1}
\right]\notag\\
&=\frac{w_{\bi_1}K^{\bbeta,\bi_1}_0(0)}{K^{\bbeta,\bs w}_0(0)}\E_{z_0}^{\bi_1}\left[\frac{\e^{\beta_{\bi_1} t} K_0^{\bbeta,\bs w}(0)}{K_0^{\bbeta,\bs w}(t)} f(\ms Z_t)\E^{\bi_1}_{\ms Z_t}[\e^{-A^{\bbeta,\bs w,\bi_1}_0(T^{\bi_1}_0)}];t<T_0^{\bi_1}\right]\notag\\
&=\E_{z_0}^{\bi_1}\left[\frac{\e^{\beta_{\bi_1} t} w_{\bi_1}K_0^{\bbeta,\bi_1}(0)}{K_0^{\bbeta,\bs w}(t)} f(\ms Z_t)\E^{\bi_1}_{\ms Z_t}[\e^{-A^{\bbeta,\bs w,\bi_1}_0(T^{\bi_1}_0)}]
;t<T_0^{\bi_1}\right]\notag\\
&=\E_{z_0}^{(0)}\left[\frac{w_{\bi_1}K_0^{\bbeta,\bi_1}(t)}{K_0^{\bbeta,\bs w}(t)} f(\ms Z_t)\E^{\bi_1}_{\ms Z_t}[\e^{-A^{\bbeta,\bs w,\bi_1}_0(T^{\bi_1}_0)}]\right],\label{proof:FK:m=0}
\end{align}
where the second equality uses the Markov property, and the last equality uses \eqref{DY:com-3}. 
Note
\begin{align}\label{eq:finalhA}
\begin{split}
\mathring{A}^{\bbeta,\bs w}_0(t)&=\mathring{A}^{\bbeta,\bs w,\bi}_0(t)=0,\quad \forall\; t\leq T_0^1,\;\bi\in \mc E_N,
\end{split}
\end{align}
where the first line follows from \eqref{def:finalhA}. Hence, \eqref{proof:FK:m=0} proves \eqref{FK:m=0}.
In particular, by summing over the right-hand side for all $\bi_1\in \EN$ and using \eqref{Pinhomo:idsum-3}, we obtain \eqref{FK:m=0+}.
\end{proof}

\begin{proof}[Proof of Theorem~\ref{thm:main2} (2$\cc$) for $\bs m\bs =\bs 1$]
By Proposition~\ref{prop:main1}, it suffices to show 
\begin{align}
\begin{split}\label{int:m=1}
&\quad\;  2K^{\bbeta,\bi_1}_0(0)\int_\Omega \d \P^{\bi_1}_{z_0}(\omega_1)\int_0^t  \d L^{\bi_1}_{\tau_1}(\omega_1)\e^{\beta_{\bi_1} \tau_1}\E^{(0)}_{\ms Z_{\tau_1}(\omega_1)}[f(\ms Z_{t-\tau_1})]\\
&= \E_{z_0}^{\bbeta,\bs w}\left[\frac{\e^{{A}^{\bbeta,\bs w}_0(t)}K_0^{\bbeta,\bs w}(0)}{K_0^{\bbeta,\bs w}(t)} f( \ms Z_t);
\left\{
\begin{array}{cc}
 T_0^1\leq t<T_0^{2},\\
\ms Z_{T_0^1}\in \CNwia
\end{array}
\right\}
\right].
\end{split}
\end{align}

We need several steps to reach \eqref{int:m=1}. First,
\begin{align*}
&\quad \; 2K^{\bbeta,\bi_1}_0(0)\int_\Omega \d \P^{\bi_1}_{z_0}(\omega_1)\int_0^t  \d  L^{\bi_1}_{\tau_1}(\omega_1)\e^{\beta_{\bi_1} \tau_1}\E^{(0)}_{\ms Z_{\tau_1}(\omega_1)}[f(\ms Z_{t-\tau_1})]\\
&= 2K^{\bbeta,\bi_1}_0(0)\E^{\bi_1}_{z_0}\left[\int_0^t (\d L^{\bi_1}_{\tau_1})\e^{\beta_{\bi_1} \tau_1}\E^{(0)}_{\ms Z_{\tau_1}}[f(\ms Z_{t-\tau_1})]\right]\\
&=2K^{\bbeta,\bi_1}_0(0)\E^{\bi_1}_{z_0}\left[\int_{T_0^{\bi_1}}^t (\d L^{\bi_1}_{\tau_1})\e^{\beta_{\bi_1} \tau_1}\E^{(0)}_{\ms Z_{\tau_1}}[f(\ms Z_{t-\tau_1})];T_0^{\bi_1}<t\right]\\
&=\frac{2K^{\bbeta,\bs w}_0(0)}{w_{\bi_1}}\frac{w_{\bi_1}K^{\bbeta,\bi_1}_0(0)}{K^{\bbeta,\bs w}_0(0)}\E^{\bi_1}_{z_0}\left[\e^{\beta_{\bi_1} T_0^{\bi_1}}
\E^{\bi_1}_{\ms Z_{T_0^{\bi_1}}}\left.\left[\int_{0}^{t-r} (\d L^{\bi_1}_{\tau_1})\e^{\beta_{\bi_1} \tau_1}\E^{(0)}_{\ms Z_{\tau_1}}[f(\ms Z_{t-r-\tau_1})]\right]\right|_{r=T_0^{\bi_1}};T_0^{\bi_1}<t\right]\\
&=\frac{2K^{\bbeta,\bs w}_0(0)}{w_{\bi_1}}\frac{w_{\bi_1}K^{\bbeta,\bi_1}_0(0)}{K^{\bbeta,\bs w}_0(0)}\E^{\bi_1}_{z_0}\left[\e^{\beta_{\bi_1} T_0^{\bi_1}}
\E^{\bi_1}_{\ms Z_{T_0^{\bi_1}}}\left.\left[\int_{0}^{t-r} (\d  L^{\bi_1}_{\tau_1})\e^{\beta_{\bi_1} \tau_1}\E^{(0)}_{\ms Z_{\tau_1}}[f(\ms Z_{t-r-\tau_1})]\right]\right|_{r=T_0^{\bi_1}};T_0^{\bi_1}\leq t\right]\\
&=\frac{2K^{\bbeta,\bs w}_0(0)}{w_{\bi_1}}
\E^{\bbeta,\bs w}_{z_0}\Bigg[\e^{\mathring{A}^{\bbeta,\bs w}_0(T_0^1)+\overline{A}^{\bbeta,\bs w}_{0}(T_0^{1})}
\E^{\bi_1}_{\ms Z_{T_0^{1}}}\left.\left[\int_{0}^{t-r} (\d L^{\bi_1}_{\tau_1})\e^{\beta_{\bi_1} \tau_1}\E^{(0)}_{\ms Z_{\tau_1}}[f(\ms Z_{t-r-\tau_1})]\right]\right|_{r=T_0^{1}};\\
&\quad\;\left\{
\begin{array}{cc}
 T_0^1\leq t,\\
\ms Z_{T_0^1}\in \CNwia
\end{array}
\right\}\Bigg],
\end{align*}
where the third equality uses conditioning on $\{T_0^{\bi_1}<t\}$, the fourth equality uses the existence of probability density of $T_0^{\bi_1}$ under $\P^{\bi_1}$ \cite[p.884]{DY:Krein-3}, 
and the last equality follows from \eqref{Pinhomo:id-3} and \eqref{eq:finalhA}. 
Continuing from the last equality, we use \eqref{exp:LT-3} with $\bi=\bi_1$ and $z_0\in \CN$ such that $z_0^{\bi_1}=0$ to get
\begin{align}
&\quad\;2K^{\bbeta,\bi_1}_0(0)\int_\Omega \d \P^{\bi_1}_{z_0}(\omega_1)\int_0^t  \d  L^{\bi_1}_{\tau_1}(\omega_1)\e^{\beta_{\bi_1} \tau_1}\E^{(0)}_{\ms Z_{\tau_1}(\omega_1)}[f(\ms Z_{t-\tau_1})]\notag\\
&=\frac{2K^{\bbeta,\bs w}_0(0)}{w_{\bi_1}}\E^{\bbeta,\bs w}_{z_0}\Bigg[\e^{\mathring{A}^{\bbeta,\bs w}_0(T_0^1)+\overline{A}^{\bbeta,\bs w}_{0}(T_0^{1})}
\E^{\bi_1}_{\ms Z_{T_0^{1}}}\left.\left[\frac{\e^{\beta_{\bi_1} (t-r)}f(\ms Z_{t-r})}{2K_0^{\bbeta,\bi_1}(t-r)}\right]\right|_{r=T_0^{1}};\left\{
\begin{array}{cc}
 T_0^1\leq t,\\
\ms Z_{T_0^1}\in\CNwia
\end{array}
\right\}\Bigg]\notag\\
&=\frac{2K^{\bbeta,\bs w}_0(0)}{w_{\bi_1}}\E^{\bbeta,\bs w}_{z_0}\Bigg[\e^{\mathring{A}^{\bbeta,\bs w}_0(T_0^1)+\overline{A}^{\bbeta,\bs w}_{0}(T_0^{1})}\notag\\
&\quad \times\E^{\bbeta,\bs w,\bi_1}_{\ms Z_{T_0^{1}}}\left.\left[\frac{\e^{\beta_{\bi_1}(t-r)}f(\ms Z_{t-r})}{2K_0^{\bbeta,\bi_1}(t-r)}\frac{\e^{A^{\bbeta,\bs w,\bi_1}_0(t-r)}w_{\bi_1}K^{\bbeta,\bi_1}_0(t-r)}{K^{\bbeta,\bs w}_0(t-r)};t-r<T_0^{\bi_1,1}\right]\right|_{r=T_0^{1}};\notag\\
&\quad \left\{
\begin{array}{cc}
 T_0^1\leq t,\\
\ms Z_{T_0^1}\in \CNwia
\end{array}
\right\}\Bigg],\label{int:m=1-1}
\end{align}
where the last equality follows from \eqref{def:Q21-3}. In more detail, we have used in \eqref{int:m=1-1} the probability measure $\P^{\bbeta,\bs w,\bi_1}_{z_1}$, $z_1\in \CNwia$, for the associated stochastic many-$\delta$ motion \cite[Theorem~3.1]{C:SDBG2-3} and the associated stopping time $T_0^{\bi_1,1}$. 
Note that by \eqref{def:A-3} and \eqref{dec:finalA}, 
\[
\beta_{\bi_1}(t-r)+A_0^{\bbeta,\bs w,\bi_1}(t-r)=\mathring{A}^{\bbeta,\bs w}_0(t-r)+\overline{A}^{\bbeta,\bs w}_0(t-r)={A}^{\bbeta,\bs w}_0(t-r)\quad\mbox{$\P^{\bbeta,\bs w,\bi_1}_{z_1}$-a.s. on }\{t-r<T_0^{\bi_1,1}\}
\]
for any $z_1\in \CNwia$. 
Hence, by the strong Markov property of $\{\ms Z_t\}$ under $\P^{\bbeta,\bs w}_{z_0}$, we obtain \eqref{int:m=1} from \eqref{int:m=1-1}. This completes the proof of Theorem~\ref{thm:main2} (2$\cc$) for $m=1$.
\end{proof}

\begin{notation}
Set $T_0^{\bj}\,\defeq\,\inf\{t\geq s;Z_t^{\bj}=0\}$ under $\P^{\bj,s}_{z_1}$, where $\P_{z_1}^{\bi,s}$ is defined in \eqref{shiftprob}.
\qed 
\end{notation}

The proof of Theorem~\ref{thm:main2} (2$\cc$) for $m\geq 2$ is more complicated, so let us briefly describe the method of proof in the simplest case of $m=2$ first. By Proposition~\ref{prop:main1}, the corresponding analytic path integral can be written as
\begin{align}\label{m=2sum}
&\quad\;\int _{0< s_{1} < s_{2}< t}\mathrm{d} \boldsymbol s_{2} P^{\bbeta;{\mathbf i}_{1},{
\mathbf i}_{2}}_{s_{1},s_{2},t}f(z_{0})\notag\\
&= 2K^{\bbeta,\bi_1}_0(0)\int_\Omega \d \P^{\bi_1}_{z_0}(\omega_1)\int_{0}^{t}  \d L^{\bi_1}_{\tau_1}(\omega_1)\e^{\beta_{\bi_1} \tau_1}\notag\\
&\quad  \times  2K^{\bbeta,\bi_2}_0(\tau_1)(\omega_1)\int_\Omega \d \P^{\bi_2,\tau_1}_{\ms Z_{\tau_{1}}(\omega_{1})}(\omega_2)\int_{\tau_{1}}^{t} \d L^{\bi_2}_{\tau_2}(\omega_2)\e^{\beta_{\bi_2} (\tau_2-\tau_{1})} \E_{\ms Z_{\tau_2}(\omega_2)}^{(0)}[f(\ms Z_{t-\tau_2})].
\end{align}
As we have adopted the notation in \eqref{shiftprob}, the right-hand side can be viewed naturally as a path integral under which a path is run under $\P^{\bi_1}_{z_0}$ until time $\tau_1$, continues with $\P^{\bi_2}_{\ms Z_{\tau_{1}}(\omega_{1}),\tau_{1}}$ until time $\tau_2$, and is completed under $\P_{\ms Z_{\tau_2}(\omega_2)}^{(0)}$ to stop at time $t$.

In the following proof, we will convert the iterated integral on the right-hand side of \eqref{m=2sum} to a Feynman--Kac-type formula by subdividing the path we just described into the following three groups of paths and then 
converting them backward in time to paths of the same stochastic many-$\delta$ motion:
\begin{itemize}
\item The path over $[T_0^{\bi_2},\tau_2]$ under $\P^{\bi_2}_{\ms Z_{\tau_1}(\omega_1),\tau_1}$ and the path over $[\tau_2,t]$ under $\P_{\ms Z_{\tau_2}(\omega_2)}^{(0)}$. 
\item The path over $[T_0^{\bi_1},\tau_1]$ under $\P^{\bi_1}_{z_0}$ and the path over $[\tau_1,T_0^{\bi_2}]$ under $\P^{\bi_2,\tau_1}_{\ms Z_{\tau_1}(\omega_1)}$.
\item The path over $[0,T_0^{\bi_1}]$ under $\P^{\bi_1}_{z_0}$. 
\end{itemize}
In particular, expectations of the form on the left-hand side of \eqref{connector:20} below will emerge from the conversions. We postpone the proof of the following proposition to Section~\ref{sec:connector}.

\begin{prop}\label{prop:connector}
Fix $\bi,\bj\in \mc E_N$ with $\bi\neq \bj$, 
 and let $z_0\in \CNwi$.
Then for all bounded $F\in \B_+(\CN\times\R_+)$ and $t\geq 0$,
\begin{align}
&\quad\;\E^{\bi}_{z_0}\left[\int_0^{t}\d  L^{\bi}_{\tau}\e^{\beta_\bi \tau}K_0^{\bbeta,\bj}(\tau)\E^{\bj,\tau}_{\ms Z_{\tau}}[\e^{\beta_{\bj}(
T_0^{\bj}-\tau
)
}F(\ms Z_{T_0^{\bj}},T_0^{\bj});T_0^{\bj}\leq t]
\right]\notag\\
\begin{split}
&=\frac{w_\bi}{w_\bj}
\E^{\bbeta,\bs w,\bi}_{z_0}\left[\frac{\e^{{A}^{\bbeta,\bs w}_0(T_0^{\bi,1})}}{2}F(\ms Z_{T_0^{\bi,1}},T_0^{\bi,1});\left\{
\begin{array}{cc}
T_0^{\bi,1}\leq t,\\
\ms Z_{T_0^{\bi,1}}\in \CNwj
\end{array}
\right\}
\right].\label{connector:20}
\end{split}
\end{align}
\end{prop}

\begin{proof}[Proof  of Theorem~\ref{thm:main2} (2$\cc$) for all $\bs m\bs \geq \bs 2$]
As in the case of $m=1$, we work with \eqref{rep:Plimit}.\medskip 

\noindent {\bf Step 1.} 
Our goal in this step is to prove the following identity: 
\begin{align}
&\quad\;\int_{0<s_1<\cdots<s_m<t}\d \bs s_m P^{\bbeta;\bi_1,\cdots,\bi_m}_{s_1,\cdots,s_m,t}f(z_0)\notag\\
&=2^mK_0^{\bbeta,\bi_1}(0)\int_\Omega \d \P^{\bi_1}_{z_0}(\omega_1)\int_0^t  \d  L^{\bi_1}_{\tau_1}(\omega_1)\e^{\beta_{\bi_1} \tau_1}\notag\\
&\quad\times K_0^{\bbeta,\bi_2}(\tau_{1})(\omega_1)\int_\Omega \d \P^{\bi_2,\tau_1}_{ \ms Z_{\tau_1}(\omega_1)}(\omega_2)\int_{\tau_1}^{t}  \d L^{\bi_2}_{\tau_2}(\omega_2)\e^{\beta_{\bi_2} (\tau_2-\tau_1)} \cdots\notag\\
&\quad \times  K_0^{\bbeta,\bi_{m-1}}(\tau_{m-2})(\omega_{m-2})\int_\Omega \d \P^{\bi_{m-1},\tau_{m-2}}_{ \ms Z_{\tau_{m-2}}(\omega_{m-2})}(\omega_{m-1})  \int_{\tau_{m-2}}^{t} \d L^{\bi_{m-1}}_{\tau_{m-1}}(\omega_{m-1})\e^{\beta_{\bi_{m-1}}( \tau_{m-1}-\tau_{m-2})}\notag\\
&\quad\times  K_0^{\bbeta,\bi_m}(\tau_{m-1})(\omega_{m-1}) \int_\Omega \d \P^{\bi_m,\tau_{m-1}}_{ \ms Z_{\tau_{m-1}}(\omega_{m-1})}(\omega_m')\e^{\beta_{\bi_m} (T_0^{\bi_m}(\omega_m')-\tau_{m-1})}\1_{\{T_0^{\bi_m}(\omega_m')\leq t\}}\notag\\
&\quad \times  F_m\left(\ms Z_{T_0^{\bi_m}}(\omega_m'),T_0^{\bi_m}(\omega_m') \right),\label{FK:5}
\end{align}
where
\begin{align*}
F_m\left(z_m,r_m\right)
\defeq\,w_{\bi_m}\E^{\bbeta,\bs w,\bi_m}_{z_m}\left[\frac{\e^{{A}^{\bbeta,\bs w}_0(t-r_m)} }{2K_0^{\bbeta,\bs w}(t-r_m)} f(\ms Z_{t-r_m}); t-r_m<T_0^{\bs w\setminus\{w_{\bi_m}\}}\right].
\end{align*}

The proof of \eqref{FK:5} is similar to the above proof of (\ref{sum:DBG}) in the case of $m = 1$. Specifically, we focus on the iterated integral over $[\tau_{m-1},t]$ from the right-hand side of \eqref{rep:Plimit}, namely,
\begin{align}
& \quad\; \int_\Omega \d \P^{\bi_m,\tau_{m-1}}_{ \ms Z_{\tau_{m-1}}(\omega_{m-1})}(\omega_m) \int_{\tau_{m-1}}^{t}  \d L^{\bi_m}_{\tau_m}(\omega_m)\e^{\beta_{\bi_m} (\tau_m-\tau_{m-1})}
 \E_{\ms Z_{\tau_m}(\omega_m)}^{(0)}[f(\ms Z_{t-\tau_m})]\notag\\
 &=  \int_\Omega \d \P^{\bi_m,\tau_{m-1}}_{ \ms Z_{\tau_{m-1}}(\omega_{m-1})}(\omega_m')\e^{\beta_{\bi_m} (T_0^{\bi_m}(\omega_m')-\tau_{m-1})}\1_{\{T_0^{\bi_m}(\omega_m')\leq t\}}\notag\\
&\quad \int_{\Omega}\d\P^{\bi_m,T_0^{\bi_m}(\omega_m')}_{\ms Z_{T_0^{\bi_m}(\omega_m')}}(\omega_m)
\int_{T_0^{\bi_m}(\omega_m')}^{t}  \d L^{\bi_m}_{\tau_m}(\omega_m)\e^{\beta_{\bi_m} (\tau_m-T_0^{\bi_m}(\omega_m'))} \E_{\ms Z_{\tau_m}(\omega_m)}^{(0)}[f(\ms Z_{t-\tau_m})],\label{FK:30}
\end{align} 
where the last equality follows by conditioning on $\{T^{\bi_m}_0(\omega'_m)\leq t\}$ under $\d \P^{\bi_m,\tau_{m-1}}_{ \ms Z_{\tau_{m-1}}(\omega_{m-1})}(\omega_m')$. To simplify the right-hand side, note that 
\begin{align}
&\quad\;\int_{\Omega}\d\P^{\bi_m,T_0^{\bi_m}(\omega_m')}_{\ms Z_{T_0^{\bi_m}(\omega_m')}}(\omega_m)
\int_{T_0^{\bi_m}(\omega_m')}^{t}  \d L^{\bi_m}_{\tau_m}(\omega_m)\e^{\beta_{\bi_m} (\tau_m-T_0^{\bi_m}(\omega_m'))} \E_{\ms Z_{\tau_m}(\omega_m)}^{(0)}[f(\ms Z_{t-\tau_m})]\notag\\
&=\left.\E^{\bi_m}_{\ms Z_{T_0^{\bi_m}(\omega_m')}}\left[
\int_{0}^{t-r_m} (\d L^{\bi_m}_{\tau_m})\e^{\beta_{\bi_m}\tau_m}\E_{\ms Z_{\tau_m}}^{(0)}[f(\ms Z_{t-r_m-\tau_m})]\right]\right|_{r_m=T_0^{\bi_m}(\omega_m')}\notag\\
&=\left.\E^{\bi_m}_{\ms Z_{T_0^{\bi_m}(\omega_m')}}\left[\frac{\e^{\beta_{\bi_m} (t-r_m)}f(\ms Z_{t-r_m})}{2K_0^{\bbeta,\bi_m}(t-r_m)}\right]\right|_{r_m=T_0^{\bi_m}(\omega_m')}\notag
\end{align}
by \eqref{exp:LT-3} with $\bi=\bi_m$ and $t$ replaced by $t-r_m$. The last expectation allows the use of \eqref{def:Q21-3} with $z_0\in \CNwi$, in which case 
 $w_\bi K^{\bbeta,\bi}_0(0)/K^{\bbeta,\bs w}_0(0)=1$ (cf. \cite[Proposition~3.7 (1$\cc$)]{C:SDBG2-3}). Hence,
\begin{align}
&\quad\;\int_{\Omega}\d\P^{\bi_m,T_0^{\bi_m}(\omega_m')}_{\ms Z_{T_0^{\bi_m}(\omega_m')}}(\omega_m)
\int_{T_0^{\bi_m}(\omega_m')}^{t}  \d L^{\bi_m}_{\tau_m}(\omega_m)\e^{\beta_{\bi_m} (\tau_m-T_0^{\bi_m}(\omega_m'))} \E_{\ms Z_{\tau_m}(\omega_m)}^{(0)}[f(\ms Z_{t-\tau_m})]\notag\\
&=\left.\E^{\bbeta,\bs w,\bi_m}_{\ms Z_{T_0^{\bi_m}(\omega_m')}}\left[\frac{\e^{\beta_{\bi_m} (t-r_m)}f(\ms Z_{t-r_m})}{2K_0^{\bbeta,\bi_m}(t-r_m)}\mathcal E^{\bbeta,\bs w,\bi_m}_{\ms Z_{T_0^{\bi_m}(\omega_m')}}(t-r_m)^{-1} ; t-r_m<T_0^{\bs w\setminus\{w_{\bi_m}\}}\right]\right|_{r_m=T_0^{\bi_m}(\omega_m')}\notag\\
&=\left.w_{\bi_m}\E^{\bbeta,\bs w,\bi_m}_{\ms Z_{T_0^{\bi_m}(\omega_m')}}\left[\frac{\e^{{A}^{\bbeta,\bs w}_0(t-r_m)} }{2K_0^{\bbeta,\bs w}(t-r_m)} f(\ms Z_{t-r_m}); t-r_m<T_0^{\bs w\setminus\{w_{\bi_m}\}}\right]\right|_{r_m=T_0^{\bi_m}(\omega_m')}.
\label{FK:3}
\end{align}
The algerba done for the last equality is as follows: since $L^{\bj}_{t-r_m}=0$ for all $\bj\neq \bi_m$ on $\{t-r_m< T_0^{\bi_m\complement}\}$, we have
\begin{align*}
A^{\bbeta,\bs w,\bi_m}_0(t-r_m)+\beta_{\bi_m}(t-r_m)&=\mathring{A}^{\bbeta,\bs w}_0(t-r_m)+\overline{A}^{\bbeta,\bs w,\bi_m}_0(t-r_m)
={A}^{\bbeta,\bs w}_0(t-r_m).
\end{align*}
Applying \eqref{FK:30} and \eqref{FK:3} to \eqref{rep:Plimit} proves \eqref{FK:5}. \medskip 

\noindent {\bf Step~2.} Recall the stopping times defined in \eqref{def:Tm2-3}, and define
\begin{align}\label{def:Fell}
F_{\ell}(z_{\ell},r_{\ell})
\,
\defeq\,w_{\bi_\ell}\E^{\bbeta,\bs w,\bi_\ell}_{z_{\ell}}\left[\frac{\e^{{A}^{\bbeta,\bs w}_0(t-r_{\ell})}f(\ms Z_{t-r_{\ell}}) }{2^{m-\ell+1}K_0^{\bbeta,\bs w}(t-r_{\ell})}  ;
\left\{
\begin{array}{cc}
T^{\bi_\ell,m-\ell}_0\leq t-r_\ell<T^{\bi_\ell,m-\ell+1}_0,\\
\ms Z_{T^{\bi_\ell,1}_0}\in  \CNwid,\cdots,\\
\ms Z_{T^{\bi_\ell,m-\ell}_0}\in \CNwib
\end{array}
\right\}
 \right]
 \end{align}
for all $z_\ell\in \CNwie$, $0\leq r_\ell\leq t$ and $1\leq \ell\leq m-1$. Our goal here is to prove that with 
\[
\tau_0=0, \quad K_0^{\bbeta,\bi_{1}}(\tau_{0})(\omega_{0})\equiv K_0^{\bbeta,\bi_1}(0),\quad \P^{\bi_{1},\tau_{0}}_{ \ms Z_{\tau_{0}}(\omega_{0})}\equiv  \P^{\bi_{1}}_{ z_0},
\]
the following identity holds for any $2\leq \ell\leq m$: 
\begin{align}
&\quad\; K_0^{\bbeta,\bi_{\ell-1}}(\tau_{\ell-2})(\omega_{\ell-2})\int_\Omega \d \P^{\bi_{\ell-1},\tau_{\ell-2}}_{ \ms Z_{\tau_{\ell-2}}(\omega_{\ell-2})}(\omega_{\ell-1})\int_{\tau_{\ell-2}}^{t}  \d L^{\bi_{\ell-1}}_{\tau_{\ell-1}}(\omega_{\ell-1})\e^{\beta_{\bi_{\ell-1}} (\tau_{\ell-1}-\tau_{\ell-2})}\notag\\
&\quad  \times  K_0^{\bbeta,\bi_\ell}(\tau_{\ell-1})(\omega_{\ell-1})\int_\Omega \d \P^{\bi_\ell,\tau_{\ell-1}}_{ \ms Z_{\tau_{\ell-1}}(\omega_{\ell-1})}(\omega_\ell')\e^{\beta_{\bi_\ell} (T_0^{\bi_\ell}(\omega_\ell')-\tau_{\ell-1})}\1_{\{T_0^{\bi_\ell}(\omega_\ell')\leq t\}}\notag\\
&\quad\times F_\ell\left(\ms Z_{T_0^{\bi_\ell}}(\omega_\ell'),T_0^{\bi_\ell}(\omega_\ell') \right)\notag\\
&= K_0^{\bbeta,\bi_{\ell-1}}(\tau_{\ell-2})(\omega_{\ell-2})\int_\Omega \d \P^{\bi_{\ell-1},\tau_{\ell-2}}_{ \ms Z_{\tau_{\ell-2}}(\omega_{\ell-2})}(\omega_{\ell-1}')\e^{\beta_{\bi_{\ell-1}} (T_0^{\bi_{\ell-1}}(\omega_{\ell-1}')-\tau_{\ell-2})}\1_{\{T^{\bi_{\ell-1}}_0(\omega_{\ell-1}')\leq t\}}\notag\\
&\quad\times F_{\ell-1}\left(\ms Z_{T_0^{\bi_{\ell-1}}}(\omega_{\ell-1}'),T_0^{\bi_{\ell-1}}(\omega_{\ell-1}') \right).\label{FK:auxconnect}
\end{align}
By  iterations, \eqref{FK:5} and \eqref{FK:auxconnect} imply the following identities for all $m\geq 3$ and $3\leq \ell\leq m$:
\begin{align}
&\quad\;\int_{0<s_1<\cdots<s_m<t}\d \bs s_m P^{\bbeta;\bi_1,\cdots,\bi_m}_{s_1,\cdots,s_m,t}f(z_0)\notag\\
&=2^mK_0^{\bbeta,\bi_1}(0)\int_\Omega \d \P^{\bi_1}_{z_0}(\omega_1)\int_0^t  \d  L^{\bi_1}_{\tau_1}(\omega_1)\e^{\beta_{\bi_1} \tau_1}\notag\\
&\quad\times K_0^{\bbeta,\bi_2}(\tau_{1})(\omega_1)\int_\Omega \d \P^{\bi_2,\tau_1}_{ \ms Z_{\tau_1}(\omega_1)}(\omega_2)\int_{\tau_1}^{t}  \d L^{\bi_2}_{\tau_2}(\omega_2)\e^{\beta_{\bi_2} (\tau_2-\tau_1)} \cdots\notag\\
&\quad \times  K_0^{\bbeta,\bi_{\ell-1}}(\tau_{\ell-2})(\omega_{\ell-2})\int_\Omega \d \P^{\bi_{\ell-1},\tau_{\ell-2}}_{ \ms Z_{\tau_{\ell-2}}(\omega_{\ell-2})}(\omega_{\ell-1})  \int_{\tau_{\ell-2}}^{t} \d L^{\bi_{\ell-1}}_{\tau_{\ell-1}}(\omega_{\ell-1})\e^{\beta_{\bi_{\ell-1}}( \tau_{\ell-1}-\tau_{\ell-2})}\notag\\
&\quad\times  K_0^{\bbeta,\bi_\ell}(\tau_{\ell-1})(\omega_{\ell-1}) \int_\Omega \d \P^{\bi_\ell,\tau_{\ell-1}}_{ \ms Z_{\tau_{\ell-1}}(\omega_{\ell-1})}(\omega_\ell')\e^{\beta_{\bi_\ell} (T_0^{\bi_\ell}(\omega_\ell')-\tau_{\ell-1})}\1_{\{T_0^{\bi_\ell}(\omega_\ell')\leq t\}}\notag\\
&\quad \times F_{\ell}\left(\ms Z_{T_0^{\bi_{\ell}}}(\omega_{\ell}'),T^{\bi_{\ell}}_0(\omega'_{\ell}) \right),\notag
\end{align}
and more generally, for all $m\geq 2$,
\begin{align}
&\quad\;\int_{0<s_1<\cdots<s_m<t}\d \bs s_m P^{\bbeta;\bi_1,\cdots,\bi_m}_{s_1,\cdots,s_m,t}f(z_0)\notag\\
&=2^mK_0^{\bbeta,\bi_1}(0)\int_\Omega \d \P^{\bi_1}_{z_0}(\omega_1)\int_0^t  \d  L^{\bi_1}_{\tau_1}(\omega_1)\e^{\beta_{\bi_1} \tau_1}\notag\\
&\quad\times  K_0^{\bbeta,\bi_2}(\tau_{1})(\omega_{1}) \int_\Omega \d \P^{\bi_2,\tau_{1}}_{ \ms Z_{\tau_{1}}(\omega_{1})}(\omega_2')\e^{\beta_{\bi_2} (T_0^{\bi_2}(\omega_2')-\tau_{1})}\1_{\{T_0^{\bi_2}(\omega_2')\leq t\}}F_{2}\left(\ms Z_{T_0^{\bi_{2}}}(\omega_{2}'),T^{\bi_{2}}_0(\omega'_{2}) \right)\notag\\
&=2^mK_0^{\bbeta,\bi_1}(0)\int_\Omega \d \P^{\bi_1}_{z_0}(\omega_1')\1_{\{T_0^{\bi_1}(\omega_1')\leq t\}} \e^{\beta_{\bi_1}T_0^{\bi_1}(\omega_1')}F_{1}\left(\ms Z_{T_0^{\bi_{1}}}(\omega_{1}'),T^{\bi_{1}}_0(\omega'_{1}) \right)\label{FK:8}.
\end{align}
We will close the proof of \eqref{sum:DBG} for all $m\geq 2$ in the next step by using \eqref{FK:8}.

To prove \eqref{FK:auxconnect} for $2\leq \ell\leq m$, consider its left-hand side and condition on $\{T_0^{\bi_\ell}(\omega_\ell')\leq t\}$: 
\begin{align}
&\quad\; K_0^{\bbeta,\bi_{\ell-1}}(\tau_{\ell-2})(\omega_{\ell-2})\int_\Omega \d \P^{\bi_{\ell-1},\tau_{\ell-2}}_{ \ms Z_{\tau_{\ell-2}}(\omega_{\ell-2})}(\omega_{\ell-1})\int_{\tau_{\ell-2}}^{t}  \d L^{\bi_{\ell-1}}_{\tau_{\ell-1}}(\omega_{\ell-1})\e^{\beta_{\bi_{\ell-1}} (\tau_{\ell-1}-\tau_{\ell-2})}\notag\\
&\quad  \times  K_0^{\bbeta,\bi_\ell}(\tau_{\ell-1})(\omega_{\ell-1})\int_\Omega \d \P^{\bi_\ell,\tau_{\ell-1}}_{ \ms Z_{\tau_{\ell-1}}(\omega_{\ell-1})}(\omega_\ell')\e^{\beta_{\bi_\ell} (T_0^{\bi_\ell}(\omega_\ell')-\tau_{\ell-1})}\1_{\{T_0^{\bi_\ell}(\omega_\ell')\leq t\}}\notag\\
&\quad\times F_\ell\left(\ms Z_{T_0^{\bi_\ell}}(\omega_\ell'),T_0^{\bi_\ell}(\omega_\ell') \right)\notag\\
 &= K_0^{\bbeta,\bi_{\ell-1}}(\tau_{\ell-2})(\omega_{\ell-2})\int_\Omega \d \P^{\bi_{\ell-1},\tau_{\ell-2}}_{ \ms Z_{\tau_{\ell-2}}(\omega_{\ell-2})}(\omega_{\ell-1}')\e^{\beta_{\bi_{\ell-1}} (T_0^{\bi_{\ell-1}}(\omega_{\ell-1}')-\tau_{\ell-2})}\1_{\{T^{\bi_{\ell-1}}_0(\omega_{\ell-1}')\leq t\}}\notag\\
&\quad \times \int_{\Omega}\d\P^{\bi_{\ell-1},T_0^{\bi_{\ell-1}}(\omega_{\ell-1}')}_{\ms Z_{T_0^{\bi_{\ell-1}}(\omega_{\ell-1}')}}(\omega_{\ell-1}) \int_{T^{\bi_{\ell-1}}_0(\omega_{\ell-1}')}^{t}  \d L^{\bi_{\ell-1}}_{\tau_{\ell-1}}(\omega_{\ell-1})\e^{\beta_{\bi_{\ell-1}} (\tau_{\ell-1}-T_0^{\bi_{\ell-1}}(\omega_{\ell-1}'))}\notag\\
&\quad \times  K_0^{\bbeta,\bi_\ell}(\tau_{\ell-1})(\omega_{\ell-1})\int_\Omega \d \P^{\bi_\ell,\tau_{\ell-1}}_{ \ms Z_{\tau_{\ell-1}}(\omega_{\ell-1})}(\omega_\ell')\e^{\beta_{\bi_\ell} (T_0^{\bi_\ell}(\omega_\ell')-\tau_{\ell-1})}\1_{\{T_0^{\bi_\ell}(\omega_\ell')\leq t\}}\notag\\
&\quad \times F_\ell\left(\ms Z_{T_0^{\bi_\ell}}(\omega_\ell'),T_0^{\bi_\ell}(\omega_\ell') \right).\label{FK:9}
\end{align}
To simplify the right-hand side, we fix the value of $T_0^{\bi_{\ell-1}}(\omega_{\ell-1}')$ by 
\begin{align}\label{Tr:fixedvalue}
r_{\ell-1}=T_0^{\bi_{\ell-1}}(\omega_{\ell-1}')
\end{align}
and rewrite the iterated integral over $[T_0^{\bi_{\ell-1}}(\omega_{\ell-1}'),t]$ in \eqref{FK:9} with
 the shorthand notation
\[
\Sigma r_\ell\,\defeq\,r_\ell+r_{\ell-1}
\]
when given $T_0^{\bs w\setminus\{w_{\bi_{\ell-1}}\}}(\omega_{\ell}')=r_\ell$. Therefore, 
\begin{align}
&\quad\; \int_\Omega \d\P^{\bi_{\ell-1},T_0^{\bi_{\ell-1}}(\omega_{\ell-1}')}_{\ms Z_{T_0^{\bi_{\ell-1}}(\omega_{\ell-1}')}}(\omega_{\ell-1})   \int_{T^{\bi_{\ell-1}}_0(\omega'_{\ell-1})}^{t}  \d L^{\bi_{\ell-1}}_{\tau_{\ell-1}}(\omega_{\ell-1})\e^{\beta_{\bi_{\ell-1}} (\tau_{\ell-1}-T_0^{\bi_{\ell-1}}(\omega_{\ell-1}'))}\notag\\
&\quad\;  \times  K_0^{\bbeta,\bi_\ell}(\tau_{\ell-1})(\omega_{\ell-1})\int_\Omega \d \P^{\bi_\ell,\tau_{\ell-1}}_{ \ms Z_{\tau_{\ell-1}}(\omega_{\ell-1})}(\omega_\ell')\e^{\beta_{\bi_\ell} (T_0^{\bi_\ell}(\omega_\ell')-\tau_{\ell-1})}\1_{\{T_0^{\bi_\ell}(\omega_\ell')\leq t\}}\notag\\
&\quad\;  \times F_\ell\left(\ms Z_{T_0^{\bi_\ell}}(\omega_\ell'),T_0^{\bi_\ell}(\omega_\ell') \right)\notag\\
&=\E^{\bi_{\ell-1}}_{\ms Z_{T_0^{\bi_{\ell-1}}(\omega_{\ell-1}')}}\Bigg[\int_0^{t-r_{\ell-1}}(\d  L^{\bi_{\ell-1}}_{\tau_{\ell-1}})\e^{\beta_\bi \tau_{\ell-1}}K_0^{\bbeta,\bi_\ell}(\tau_{\ell-1})\notag\\
&\quad \E^{\bi_\ell,\tau_{\ell-1}}_{\ms Z_{\tau_{\ell-1}}}\left[\e^{\beta_{\bi_\ell}(
T_0^{\bi_\ell}-\tau_{\ell-1}
)
}F_\ell\left(\ms Z_{T_0^{\bi_\ell}},T_0^{\bi_\ell}+r_{\ell-1}\right);T_0^{\bi_\ell}\leq t-r_{\ell-1}\right]
\Bigg]\notag\\
&=\frac{w_{\bi_{\ell-1}}}{w_{\bi_\ell}}\E^{\bbeta,\bs w,\bi_{\ell-1}}_{\ms Z_{T_0^{\bi_{\ell-1}}(\omega_{\ell-1}')}}\left[\frac{\e^{{A}^{\bbeta,\bs w}_0(T_0^{\bs w\setminus\{w_{\bi_{\ell-1}}\}})}}{2}w_{\bi_\ell}\right.\notag\\
&\quad \left. \E^{\bbeta,\bs w,\bi_\ell}_{\ms Z_{T_0^{\bs w\setminus\{w_{\bi_{\ell-1}}\}}} }\left[\frac{
\e^{{A}^{\bbeta,\bs w}_0(t-\Sigma r_\ell)}f(\ms Z_{t-\Sigma r_\ell})
}{2^{m-\ell+1}K_0^{\bbeta,\bs w}(t-\Sigma r_\ell)} ;  \left\{
\begin{array}{cc}
T^{\bi_\ell,m-\ell}_0\leq t-\Sigma r_\ell<T^{\bi_\ell,m-\ell+1}_0,\\
\ms Z_{T^{\bi_\ell,1}_0}\in  \CNwid,\cdots,\\
\ms Z_{T^{\bi_\ell,m-\ell}_0}\in \CNwib
\end{array}
\right\}\right]\right|_{r_\ell=T_0^{\bs w\setminus\{w_{\bi_{\ell-1}}\}}};\notag\\
&\quad\left.\left\{
\begin{array}{cc}
T_0^{\bs w\setminus\{w_{\bi_{\ell-1}}\}}\leq t- r_{\ell-1},\\
\ms Z_{T_0^{\bs w\setminus\{w_{\bi_{\ell-1}}\}}}\in \CNwie
\end{array}
\right\}
\right]\notag
\end{align}
by Proposition~\ref{prop:connector} and the definition \eqref{def:Fell} of $F_\ell$. Also, by the strong Markov property, we obtain from the last equality that
\begin{align}
&\quad\; \int_\Omega \d\P^{\bi_{\ell-1},T_0^{\bi_{\ell-1}}(\omega_{\ell-1}')}_{\ms Z_{T_0^{\bi_{\ell-1}}(\omega_{\ell-1}')}}(\omega_{\ell-1})   \int_{T^{\bi_{\ell-1}}_0(\omega'_{\ell-1})}^{t}  \d L^{\bi_{\ell-1}}_{\tau_{\ell-1}}(\omega_{\ell-1})\e^{\beta_{\bi_{\ell-1}} (\tau_{\ell-1}-T_0^{\bi_{\ell-1}}(\omega_{\ell-1}'))}\notag\\
&\quad\;  \times  K_0^{\bbeta,\bi_\ell}(\tau_{\ell-1})(\omega_{\ell-1})\int_\Omega \d \P^{\bi_\ell,\tau_{\ell-1}}_{ \ms Z_{\tau_{\ell-1}}(\omega_{\ell-1})}(\omega_\ell')\e^{\beta_{\bi_\ell} (T_0^{\bi_\ell}(\omega_\ell')-\tau_{\ell-1})}\1_{\{T_0^{\bi_\ell}(\omega_\ell')\leq t\}}\notag\\
&\quad\;  \times F_\ell\left(\ms Z_{T_0^{\bi_\ell}}(\omega_\ell'),T_0^{\bi_\ell}(\omega_\ell') \right)\notag\\
&=w_{\bi_{\ell-1}}\E^{\bbeta,\bs w,\bi_{\ell-1}}_{\ms Z_{T_0^{\bi_{\ell-1}}(\omega_{\ell-1}')}}\left[\frac{
\e^{{A}^{\bbeta,\bs w}_0(t-r_{\ell-1})}f(\ms Z_{t-r_{\ell-1}})
}{2^{m-\ell+2}K_0^{\bbeta,\bs w}(t-r_{\ell-1})}  ;\left\{
\begin{array}{cc}
T^{\bi_{\ell-1},m-\ell+1}_0\leq t-r_{\ell-1}<T^{\bi_{\ell-1},m-\ell+2}_0,\\
\ms Z_{T^{\bi_{\ell-1},1}_0}\in  \CNwie,\cdots,\\
\ms Z_{T^{\bi_{\ell-1},m-\ell+1}_0}\in \CNwib
\end{array}
\right\}
 \right]\notag\\
&= F_{\ell-1}\left(\ms Z_{T_0^{\bi_{\ell-1}}}(\omega_{\ell-1}'),T_0^{\bi_{\ell-1}}(\omega_{\ell-1}') \right),\notag
\end{align}
where the last equality holds by using the value of $r_{\ell-1}$ from \eqref{Tr:fixedvalue} and the definition \eqref{def:Fell} for $F_{\ell-1}$. Applying the last equality to \eqref{FK:9} proves \eqref{FK:auxconnect}.\medskip 

\noindent {\bf Step 3.}
Recall the stopping times defined in \eqref{def:Tm1-3}. 
Then, by \eqref{Pinhomo:id-3} and the strong Markov property at $T_0^{1}$ under $\P_{z_0}^{\bbeta,\bs w}$, \eqref{FK:8} gives
\begin{align*}
\int_{0<s_1<\cdots\atop <s_m<t}\d \bs s_m P^{\bbeta;\bi_1,\cdots,\bi_m}_{s_1,\cdots,s_m,t}f(z_0)
 =\E^{\bbeta,\bs w}_{z_0}\left[\frac{ \e^{{A}^{\bbeta,\bs w}_0(t)}K^{\bbeta,\bs w}_0(0) }{K_0^{\bbeta,\bs w}(t)} f(\ms Z_{t}) ;
\left\{
\begin{array}{cc}
T^m_0\leq t<T^{m+1}_0,\\
\ms Z_{T^1_0}\in \CNwia,\cdots,\\
\ms Z_{T^m_0}\in  \CNwib
\end{array}
\right\}\right],
\end{align*}
which is \eqref{sum:DBG} for $m\geq 2$. The proof of Theorem~\ref{thm:main2} (2$\cc$) is complete.
\end{proof}

\begin{proof}[Proof of Theorem~\ref{thm:main2} (3$\cc$)]
By \eqref{Pinhomo:idsum-3},
the required formulas in \eqref{goal:FK} follow upon using  \eqref{FK:m=0+} and \eqref{sum:DBG}.
\end{proof}

\subsection{Transformations of the intermediate parts}\label{sec:connector}
In this subsection, we prove Proposition~\ref{prop:connector} by using semi-discrete approximations. To set up these approximations, we define
\begin{align}\label{def:an}
a_n=a_n(N_0)\,\defeq\,\frac{nt}{2^{N_0}},\quad n\in \{0,1,\cdots,2^{N_0}\},\quad N_0\in \Bbb N.
\end{align}
Note that $\{a_n(N_0);n\in \{0,1,\cdots,2^{N_0}\}\}$ is increasing in $N_0$. Also,
\begin{align}
\forall\;\tau\in (0,t],\quad \1_{(\tau,t]}(T_0^{\bj})
&=\sum_{n=1}^{2^{N_0}}\1_{(a_{n-1},a_n]\cap (\tau,t]}(T_0^{\bj})=
\sum_{n=1}^{2^{N_0}}\1_{\{\tau\leq a_n\}}\1_{(a_{n-1}\vee \tau,a_n]}(T_0^{\bj}),\label{discrete}
\end{align}
and whenever $\bj\neq \bi$,
\begin{align}\label{connector:aux1}
\E^{\bi}_{z_0}\left[\int_0^\infty \d L^{\bi}_{\tau}\1_{\{Z^{\bj}_{\tau}=0\}}\right]=0
\end{align}
by \eqref{def:SDE2-3} so that 
\begin{align}\label{connector:aux2}
\E_{z_0}^{\bi}\left[\int_0^{t}\d  L^{\bi}_{\tau}\e^{\beta_\bi \tau}K_0^{\bbeta,\bj}(\tau)\E^{\bj,\tau}_{\ms Z_{\tau}}\left[\e^{\beta_\bj (T_0^{\bj}-\tau)}F(\ms Z_{T_0^{\bj}},T_0^{\bj});T_0^{\bj}=\tau\right]\right]=0.
\end{align}
By \eqref{discrete} and \eqref{connector:aux2}, the following semi-discretization of the left-hand side of \eqref{connector:20} holds:
\begin{align}
&\quad \;\E_{z_0}^{\bi}\left[\int_0^{t}\d  L^{\bi}_{\tau}\e^{\beta_\bi \tau}K_0^{\bbeta,\bj}(\tau)\E^{\bj,\tau}_{\ms Z_{\tau}}\left[\e^{\beta_\bj (T_0^{\bj}-\tau)}F(\ms Z_{T_0^{\bj}},T_0^{
\bj});T_0^{\bj}\leq  t\right]
\right]\notag\\
\begin{split}\label{conn:approx}
&=\sum_{n=1}^{2^{N_0}}\E_{z_0}^{\bi}\Bigg[\int_0^{a_n}\d  L^{\bi}_{\tau}\e^{\beta_\bi \tau}K_0^{\bbeta,\bj}(\tau)\E^{\bj,\tau}_{\ms Z_{\tau}}\left[\e^{\beta_\bj (T_0^{\bj}-\tau)}F(\ms Z_{T_0^{
\bj}},T_0^{
\bj});a_{n-1}\vee \tau<T_0^{\bj}\leq a_n\Bigg]
\right].
\end{split}
\end{align}

To use \eqref{conn:approx}, we need some perturbations of time for the right-hand side of \eqref{conn:approx}. The specific forms of the perturbations are in the following lemma. 

\begin{lem}\label{lem:connectorapprox}
Fix $t\in (0,\infty)$, $\bi\neq \bj$, and let $z_0\in \CNwi$.\medskip

\noindent {\rm (1$\cc$)} It holds that 
\begin{align}\label{connector:lim}
\lim_{\delta\searrow 0}\sup_{0\leq \tau\leq t}\E_{z_0}^{\bi}\left[\int_{0}^{t}\d  L^{\bi}_{\tau}K_0^{\bbeta,\bj}(\tau)\P^{\bj}_{\ms Z_{\tau}}(\tau< T_0^{\bj}\leq\tau+\delta)\right]=0.
\end{align}

\noindent {\rm (2$\cc$)} It holds that
\begin{align}
&\quad \;\E_{z_0}^{\bi}\left[\int_0^{t}\d  L^{\bi}_{\tau}\e^{\beta_\bi \tau}K_0^{\bbeta,\bj}(\tau)\E^{\bj,\tau}_{\ms Z_{\tau}}\left[\e^{\beta_\bj (T_0^{\bj}-\tau)}F(\ms Z_{T_0^{\bj}},T_0^{
\bj});T_0^{\bj}\leq  t\right]
\right]\notag\\
\begin{split}
&=\lim_{N_0\to\infty}\sum_{n=2}^{2^{N_0}} \E_{z_0}^{\bi}\Bigg[\int_0^{a_{n-1}}\d  L^{\bi}_{\tau}\e^{\beta_\bi \tau}K_0^{\bbeta,\bj}(\tau)\\
&\quad \times\E^{\bj,\tau}_{\ms Z_{\tau}}\left[\e^{\beta_\bj (a_{n-1}-\tau)}\E^{\bj}_{\ms Z_{a_{n-1}}}\left[\e^{\beta_\bj T_0^{\bj}}F(\ms Z_{T_0^{
\bj}},T_0^{
\bj});T_0^{\bj}\leq a_1\right];a_{n-1}<T_0^{\bj}\right]\Bigg]
\label{hope}
\end{split}
\end{align}
for all bounded $F\in \B_+(\CN\times [0,\infty))$.
\end{lem}
\begin{proof}
\noindent {\bf (1$\cc$)} Note that under $\P^{\bi}_{z_0}$ for any $z_0\in \CN$,
the equations of $\{Z^\bj_t\}$ from \eqref{def:SDE2-3} for $\bj\neq \bi$ can be written as follows where $\bi=(i\prime,i)$ and $\bj=(j\prime,j)$: 
\begin{equation}
\begin{cases}\label{SDE:Z}
\hfill Z^{(j,i\prime)}_t\displaystyle 
=\left(Z_0^{(j,i\prime)}+\frac{Z_0^{\bi}}{2}\right)-\frac{Z^{\bi}_t}{2}+
\frac{\two W^j_t-W^{\bi\prime}_t}{2},& j>i\prime,\\
\vspace{-.4cm}\\
\hfill  \displaystyle Z^{(i\prime,j)}_t\displaystyle =\left(Z_0^{(i\prime,j)}-\frac{Z_0^{\bi}}{2}\right)+\frac{Z^{\bi}_t}{2}+
\frac{W^{\bi\prime}_t-\two W^j_t}{2},& j<i\prime,\;j\neq i,\\
\vspace{-.4cm}\\
\hfill  \displaystyle Z^{(j,i)}_t\displaystyle 
=\left(Z_0^{(j,i)}-\frac{Z_0^{\bi}}{2}\right)+\frac{Z^{\bi}_t}{2}+
\frac{\two W^j_t-W^{\bi\prime}_t}{2},& j>i,\;j\neq i\prime,\\
\vspace{-.4cm}\\
\hfill  \displaystyle Z^{(i,j)}_t\displaystyle 
=\left(Z_0^{(i,j)}+\frac{Z_0^{\bi}}{2}\right)-\frac{Z^{\bi}_t}{2}+
\frac{W^{\bi\prime}_t-\two W^j_t}{2}, &j<i,\\
\vspace{-.4cm}\\
\hfill \displaystyle Z^{\bj}_t=Z^\bj_0+\frac{W^{j\prime}_t-W^j_t}{\two},& \bj\cap\bi=\varnothing. 
\end{cases}
\end{equation}
For example, when $j>i\prime$, 
\begin{align*}
Z^{(j,i\prime)}_t&=\frac{Z^j_t-Z^{i\prime}_t}{\two}=\frac{(z_0^j+W^j_t)-\frac{(z_0^{\bi\prime}+W_t^{\bi\prime})+Z^\bi_t}{\two}}{\two}=\left(Z_0^{(j,i\prime)}+\frac{Z_0^{\bi}}{2}\right)-\frac{Z^{\bi}_t}{2}+
\frac{\two W^j_t-W^{\bi\prime}_t}{2}.
\end{align*}

For the proof of Lemma~\ref{lem:connectorapprox} (1$\cc$), we only use \eqref{SDE:Z} for $z_0\in \CNwi$.
It suffices to show, for any $t_1>t$,  
\begin{align}\label{EL:continuity}
s\mapsto \E^{\bi}_{z_0}\left[\int_0^t \d L^{\bi}_\tau K_0^{\bbeta,\bj}(\tau)\P^{\bj}_{\ms Z_{\tau}}(T_0^\bj\leq s)\right]\in \C([0,t_1]).
\end{align}
Moreover, by the dominated convergence theorem, \eqref{connector:aux1} and the continuity of $t\mapsto \P^{\bj}_{z_1}(T_0^{\bj}\leq t)$ whenever $z^1\in \CN$ satisfies  $ z^\bj_1\neq 0$ \cite[(2.9)]{DY:Krein-3}, \eqref{EL:continuity} holds as soon as we have
\begin{align}\label{EL:continuity1}
\E^{\bi}_{z_0}\left[\int_0^t \d L^{\bi}_\tau K_0^{\bbeta,\bj}(\tau)\right]<\infty.
\end{align}

To obtain \eqref{EL:continuity1}, we use \eqref{SDE:Z} and the condition $z^\bi_0=0$ by $z_0\in \CNwi$
that whenever $w_\bj>0$ and $Z^{\bi}_{\tau}=0$,
$K_0^{\bbeta,\bj}(\tau)
=K_0(\sqrt{2\beta_\bj}|z_0^{\bj}+\sigma \widetilde{W}_{\tau}|)$.
Here, $\sigma>0$ is a constant, $\widetilde{W}$ is a two-dimensional standard Brownian motion with $\widetilde{W}_0=0$ and independent of $\{Z_t^{\bi}\}$, and $z_0^{\bj}\neq 0$ by the assumption $z_0\in \CNwi$ and $w_\bj>0$. 
Therefore, by the asymptotic representations  \eqref{K00-3}  and \eqref{K0infty-3} of $K_0(x)$ as $x\to\infty$ and as $x\to 0$, we get, for all $\bi\neq \bj$ with $w_\bj>0$, 
\begin{align}
&\quad\; \E^{\bi}_{z_0}\left[\int_0^t \d L^{\bi}_\tau K_0^{\bbeta,\bj}(\tau)\right]\notag\\
&\leq C(z_0,\beta_\bj)\E_{z_0}^{\bi}\biggl[\int_0^t \d L^{\bi}_\tau\biggl(
\E\biggl[\big|\log |z_0^{\bj}+\sigma \widetilde{W}_{\tau}|\big|;|z_0^{\bj}+\sigma \widetilde{W}_{\tau}|\leq \frac{|z_0^{\bj}|}{2}\biggr]+1\biggr)\biggr]<\infty.\label{LT:continuity1}
\end{align}
Here, the last inequality holds since $\E^\bi_{z_0}[L^\bi_t]<\infty$ by \cite[Theorem~2.1 (3$\cc$) and Proposition~4.5 (2$\cc$)]{C:SDBG1-3} and
\begin{align*}
\E\biggl[\big|\log |z_0^{\bj}+\sigma \widetilde{W}_{\tau}|\big|;|z_0^{\bj}+\sigma \widetilde{W}_{\tau}|\leq \frac{|z_0^{\bj}|}{2}\biggr]&=\int_{|z|\leq |z_0^{\bj}|/2}\big|\log |z|\big|\cdot \frac{1}{2\pi\sigma \tau}\exp\biggl(-\frac{|z-z_0^{\bj}|^2}{2\sigma\tau}\biggr)\d z\\
&\leq \int_{|z|\leq |z_0^{\bj}|/2}\big|\log |z|\big|\cdot \frac{1}{2\pi \sigma \tau}\exp\biggl(-\frac{|z_0^{\bj}|^2}{8\sigma\tau}\biggr)\d z,
\end{align*}
which is bounded in $\tau\in (0,\infty)$ by $z_0^{\bj}\neq 0$. 
We have proved Lemma~\ref{lem:connectorapprox} (1$\cc$).\medskip 

\noindent {\bf (2$\cc$)} We start by rewriting \eqref{conn:approx}:
\begin{align}
&\quad \;\E_{z_0}^{\bi}\left[\int_0^{t}\d  L^{\bi}_{\tau}\e^{\beta_\bi \tau}K_0^{\bbeta,\bj}(\tau)\E^{\bj,\tau}_{\ms Z_{\tau}}\left[\e^{\beta (T_0^{\bj}-\tau)}F(\ms Z_{T_0^{\bj}},T_0^{
\bj});T_0^{\bj}\leq  t\right]
\right]\notag\\
&=\sum_{n=1}^{2^{N_0}}\E_{z_0}^{\bi}\biggl[\int_0^{a_{n-1}}\d  L^{\bi_\bi}_{\tau}\e^{\beta_\bi \tau}K_0^{\bbeta,\bj}(\tau)\E^{\bj,\tau}_{\ms Z_{\tau}}\biggl[\e^{\beta_\bj (T_0^\bj-\tau)}F(\ms Z_{T_0^{
\bj}},T_0^{
\bj});a_{n-1}<T_0^{\bj}\leq a_n\biggr]
\biggr]\notag\\
&\quad+\sum_{n=1}^{2^{N_0}}R_{n},\label{connector:approx4}
\end{align}
where
\begin{align*}
R_{n}&\,\defeq\,\E_{z_0}^{\bi}\biggl[\int_{a_{n-1}}^{a_n}\d  L^{\bi}_{\tau}\e^{\beta_\bi \tau}K_0^{\bbeta,\bj}(\tau)\E^{\bj,\tau}_{\ms Z_{\tau}}\biggl[\e^{\beta_\bj (T_0^{\bj}-\tau)}F(\ms Z_{T_0^{
\bj}},T_0^{
\bj});a_{n-1}\vee \tau<T_0^{\bj}\leq a_n\biggr]
\biggr],
\end{align*}
and the equality in \eqref{connector:approx4} holds also because $a_{n-1}\vee\tau=a_{n-1}$ for all $0\leq \tau\leq a_{n-1}$. 
These terms $R_{n}$ satisfy the following bounds:
\begin{align*}
|R_{n}|&\leq C(t,F,\beta_\bi,\beta_\bj)\E_{z_0}^{\bi}\biggl[\int_{a_{n-1}}^{a_n}\d  L^{\bi}_{\tau}K_0^{\bbeta,\bj}(\tau)\P^{\bj,\tau}_{\ms Z_{\tau}}\left( \tau<T_0^{\bj}\leq \tau+\frac{t}{2^{N_0}}\right)\biggr].
\end{align*}
Hence, by (1$\cc$),
$\lim_{N_0\to\infty}\sum_{n=1}^{2^{N_0}}R_{n}=0$.

Let us finish off the proof of (2$\cc$). 
By \eqref{connector:approx4}, we obtain
\begin{align*}
&\quad\;\E_{z_0}^{\bi}\left[\int_0^{t}\d  L^{\bi}_{\tau}\e^{\beta_\bi \tau}K_0^{\bbeta,\bj}(\tau)\E^{\bj,\tau}_{\ms Z_{\tau}}\left[\e^{\beta_\bj (T_0^{\bj}-\tau)}F(\ms Z_{T_0^{\bj}},T_0^{
\bj});T_0^{\bj}\leq  t\right]
\right]\\
&=\lim_{N_0\to\infty}\sum_{n=1}^{2^{N_0}}\E_{z_0}^{\bi}\biggl[\int_0^{a_{n-1}}\d  L^{\bi_\bi}_{\tau}\e^{\beta_\bi \tau}K_0^{\bbeta,\bj}(\tau)\E^{\bj,\tau}_{\ms Z_{\tau}}\biggl[\e^{\beta_\bj (T_0^\bj-\tau)}F(\ms Z_{T_0^{
\bj}},T_0^{
\bj});a_{n-1}<T_0^{\bj}\leq a_n\biggr]
\biggr]\\
&=\lim_{N_0\to\infty}\sum_{n=2}^{2^{N_0}}\E_{z_0}^{\bi}\biggl[\int_0^{a_{n-1}}\d  L^{\bi_\bi}_{\tau}\e^{\beta_\bi \tau}K_0^{\bbeta,\bj}(\tau)\E^{\bj,\tau}_{\ms Z_{\tau}}\biggl[\e^{\beta_\bj (T_0^\bj-\tau)}F(\ms Z_{T_0^{
\bj}},T_0^{
\bj});a_{n-1}<T_0^{\bj}\leq a_n\biggr]
\biggr]\\
&=\lim_{N_0\to\infty}\sum_{n=2}^{2^{N_0}} \E_{z_0}^{\bi}\biggl[\int_0^{a_{n-1}}\d  L^{\bi}_{\tau}\e^{\beta_\bi \tau}K_0^{\bbeta,\bj}(\tau)\\
&\quad \times\E^{\bj,\tau}_{\ms Z_{\tau}}\biggl[\e^{\beta_\bj (a_{n-1}-\tau)}\E^{\bj}_{\ms Z_{a_{n-1}}}\left[\e^{\beta_\bj T_0^{\bj}}F(\ms Z_{T_0^{
\bj}},T_0^{
\bj});T_0^{\bj}\leq a_1\right];a_{n-1}<T_0^{\bj}\biggr]\biggr].
\end{align*}
Here, we drop the summand indexed by $n=1$ to get the second equality, which can be justified by using (1$\cc$); the third equality uses the Markov property.  
The proof of (2$\cc$) is complete. 
\end{proof}

\begin{proof}[Proof of Proposition~\ref{prop:connector}]
By Lemma~\ref{lem:connectorapprox} (2$\cc$), it suffices to prove that 
\begin{align}
&\quad\lim_{N_0\to\infty}\sum_{n=2}^{2^{N_0}} \E_{z_0}^{\bi}\biggl[\int_0^{a_{n-1}}\d  L^{\bi}_{\tau}\e^{\beta_\bi \tau}K_0^{\bbeta,\bj}(\tau)\notag\\
&\quad \times\E^{\bj,\tau}_{\ms Z_{\tau}}\left[\e^{\beta_\bj (a_{n-1}-\tau)}\E^{\bj}_{\ms Z_{a_{n-1}}}\left[\e^{\beta_\bj T_0^{\bj}}F(\ms Z_{T_0^{
\bj}},T_0^{
\bj});T_0^{\bj}\leq a_1\right];a_{n-1}<T_0^{\bj}\right]\biggr]\notag\\
&=\frac{w_\bi}{w_\bj}
\E^{\bbeta,\bs w,\bi}_{z_0}\Bigg[\frac{\e^{{A}^{\bbeta,\bs w}_0(T_0^{\bi,1})}}{2}F(\ms Z_{T_0^{\bi,1}},T_0^{\bi,1});\left\{
\begin{array}{cc}
T_0^{\bi,1}\leq t,\\
\ms Z_{T_0^{\bi,1}}\in \CNwj
\end{array}
\right\}
\Bigg].\label{goal:connector}
\end{align}

Since $\P^\bj(Z^\bk_s\neq 0,\;\forall\;\bk\in \mc E_N)=1$ for all $s>0$ by the existence of the probability density of $Z^\bj_s$ \cite[Theorem~2.1 (1$\cc$)]{C:SDBG1-3}, 
by \eqref{DY:com-3},
the sums on the left-hand side of \eqref{goal:connector} satisfies 
\begin{align}
&\;\quad \sum_{n=2}^{2^{N_0}} \E_{z_0}^{\bi}\Bigg[\int_0^{a_{n-1}}\d  L^{\bi}_{\tau}\e^{\beta_\bi \tau}K_0^{\bbeta,\bj}(\tau)\notag\\ &\quad\times \E^{\bj,\tau}_{\ms Z_{\tau}}\left[\e^{\beta_\bj (a_{n-1}-\tau)}\E^{\bj}_{\ms Z_{a_{n-1}}}\left[\e^{\beta_\bj T_0^{\bj}}F(\ms Z_{T_0^{
\bj}},T_0^{\bj});T_0^{\bj}\leq a_1\right];a_{n-1}<T_0^{\bj}\right]\Bigg]\notag\\
&= \sum_{n=2}^{2^{N_0}} \E_{z_0}^{\bi}\left[\int_0^{a_{n-1}}\d  L^{\bi}_{\tau}\e^{\beta_\bi \tau}K_0^{\bbeta,\bj}(\tau)\E^{(0)}_{\ms Z_{\tau}}\left[\frac{K_0^{\bbeta,\bj}(a_{n-1}-\tau)}{K^{\bbeta,\bj}_0(0)}\E^{\bj}_{ \ms Z_{a_{n-1}-\tau}}\big[\e^{\beta_\bj T_0^{\bj}}F(\ms Z_{T_0^{
\bj}},T_0^{\bj});T_0^{\bj}\leq a_1\big]\right]\right]\notag\\
&= \sum_{n=2}^{2^{N_0}} \E_{z_0}^{\bi}\left[\int_0^{a_{n-1}}\d  L^{\bi}_{\tau}\e^{\beta_\bi \tau}\E^{(0)}_{\ms Z_{\tau}}\left[K_0^{\bbeta,\bj}(a_{n-1}-\tau)\E^{\bj}_{\ms Z_{a_{n-1}-\tau}}\big[\e^{\beta_\bj T_0^{\bj}}F(\ms Z_{T_0^{
\bj}},T_0^{\bj});T_0^{\bj}\leq a_1\big]\right]\right]\notag\\
&= \sum_{n=2}^{2^{N_0}} \E_{z_0}^{\bi}\left[\frac{\e^{\beta_\bi a_{n-1}}}{2K_0^{\bbeta,\bi}(a_{n-1})}K_0^{\bbeta,\bj}(a_{n-1})
\E^{\bj}_{\ms Z_{a_{n-1}}}\big[\e^{\beta_\bj T_0^{\bj}}F(\ms Z_{T_0^{
\bj}},T_0^{\bj});T_0^{\bj}\leq a_1\big]\right],\label{connector:int}
\end{align}
where the second equality performs the cancellation of $K^{\bbeta,\bj}_0(\tau)$ under $\P^{\bi}_{z_0}$ and
$ 1/K^{\bbeta,\bj}_0(0)$ under $\P^{(0)}_{\ms Z_\tau}$, which can be justified by \eqref{connector:aux1}, and the last equality uses
 \eqref{exp:LT-3} in the case of $z^{\bi}_0=0$.

The next step is to transform all the expectations on the right-hand side of \eqref{connector:int} into expectations under $\P^{\bbeta,\bs w,\bi}$. To this end, we write the right-hand side of \eqref{connector:int}  as 
\begin{align}
&\;\quad  \sum_{n=2}^{2^{N_0}} \E_{z_0}^{\bi}\left[\frac{\e^{\beta_\bi a_{n-1}}}{2K_0^{\bbeta,\bi}(a_{n-1})}K_0^{\bbeta,\bj}(a_{n-1})
\E^{\bj}_{\ms Z_{a_{n-1}}}\big[\e^{\beta_\bj T_0^{\bj}}F(\ms Z_{T_0^{\bj}},T_0^{\bj});T_0^{\bj}\leq a_1\big]\right]\notag\\
&= \sum_{n=2}^{2^{N_0}} \E_{z_0}^{\bi}\biggl[\frac{\e^{\beta_\bi a_{n-1}}}{2K_0^{\bbeta,\bi}(a_{n-1})}K_0^{\bbeta,\bj}(a_{n-1})
\frac{K_0^{\bbeta,\bs w}(a_{n-1})}{w_\bj K_0^{\bbeta,\bj}(a_{n-1})} \E^{\bbeta,\bs w}_{\ms Z_{a_{n-1}}}\left[\e^{{A}^{\bbeta,\bs w}_0(T_0^\bj)}F(\ms Z_{T_0^{\bj}},T_0^{\bj});T_0^{\bs w}=T_0^{\bj}\leq a_1\right]\biggr]\notag
\end{align}
which follows from \eqref{Pinhomo:id-3}
since for any fixed $s>0$, $\P^\bi(\ms Z_{s}\in \CNw)=1$, and we have also used \eqref{def:A-3} and \eqref{dec:finalA}. After some cancellations, the last equality gives
\begin{align}
&\;\quad  \sum_{n=2}^{2^{N_0}} \E_{z_0}^{\bi}\left[\frac{\e^{\beta_\bi a_{n-1}}}{2K_0^{\bbeta,\bi}(a_{n-1})}K_0^{\bbeta,\bj}(a_{n-1})
\E^{\bj}_{\ms Z_{a_{n-1}}}\big[\e^{{A}^{\bbeta,\bs w}_0(T_0^\bj)}F(\ms Z_{T_0^{\bj}},T_0^{\bj});T_0^{\bj}\leq a_1\big]\right]\notag\\
&= \frac{1}{w_\bj}\sum_{n=2}^{2^{N_0}} \E_{z_0}^{\bi}\biggl[
\frac{\e^{\beta_\bi a_{n-1}}K_0^{\bbeta,\bs w}(a_{n-1})}{2K_0^{\bbeta,\bi}(a_{n-1})} \E^{\bbeta,\bs w}_{\ms Z_{a_{n-1}}}\left[\e^{{A}^{\bbeta,\bs w}_0(T_0^\bj)}F(\ms Z_{T_0^{\bj}},T_0^{\bj});T_0^{\bs w}=T_0^{\bj}\leq a_1\right]\biggr]\notag\\
&=\frac{w_\bi}{w_\bj} \sum_{n=2}^{2^{N_0}}
 \E_{z_0}^{\bbeta,\bs w,\bi}\biggl[\frac{\e^{{A}^{\bbeta,\bs w}_0(a_{n-1})}}{2}\E^{\bbeta,\bs w}_{\ms Z_{a_{n-1}}}\left[\e^{{A}^{\bbeta,\bs w}_0(T_0^\bj)}F(\ms Z_{T_0^{\bj}},T_0^{\bj});T_0^{\bs w}=T_0^{\bj}\leq a_1\right];a_{n-1}<T_0^{\bi,1}\biggr],\notag
 \end{align}
 where the last equality applies \eqref{def:Q21-3} with
$z_0\in \CNwi$, in which case $w_\bi K^{\bbeta,\bi}_0(0)/K^{\bbeta,\bs w}_0(0)$ is defined to be $1$, and again, we have also used \eqref{def:A-3} and \eqref{dec:finalA}. To continue from the last equality, 
we use the Markov property at time $a_{n-1}$ and the fact that $a_{n-1}+a_1=a_n$ for all $2\leq n\leq 2^{N_0}$ so that 
 \begin{align}
 &\;\quad  \sum_{n=2}^{2^{N_0}} \E_{z_0}^{\bi}\left[\frac{\e^{\beta_\bi a_{n-1}}}{2K_0^{\bbeta,\bi}(a_{n-1})}K_0^{\bbeta,\bj}(a_{n-1})
\E^{\bj}_{\ms Z_{a_{n-1}}}\big[\e^{\beta_\bj T_0^{\bj}}F(\ms Z_{T_0^{\bj}},T_0^{\bj});T_0^{\bj}\leq a_1\big]\right]\notag\\
 \begin{split}
&= \frac{w_\bi}{w_\bj}\sum_{n=2}^{2^{N_0}}
 \E_{z_0}^{\bbeta,\bs w,\bi}\Bigg[\frac{\e^{{A}^{\bbeta,\bs w}_0(T_0^{\bi,1})}}{2}F(\ms Z_{T_0^{\bi,1}},T_0^{\bi,1}); \left\{
 \begin{array}{cc}
 a_{n-1}<T_0^{\bi,1}=T_0^{\bj}\leq a_{n},\\
 Z^{\bi}_t\neq 0,\forall\;t\in [a_{n-1},T_0^{\bi,1}]
 \end{array}
 \right\}
 \Bigg].\label{connector:int1}
 \end{split}
\end{align}
Here, we use the event $\{Z^{\bi}_t\neq 0,\;\forall\;t\in [a_{n-1},T_0^{\bi,1}]\}$  to make up the difference between $T_0^{\bi\complement}=T_0^{\bi,1}$ and $T_0^{\mc E_N}$ when applying the Markov property. Note that since $\P^{\bbeta,\bs w,\bi}_{z_0}(T_0^1>0)=1$, we can add the term with $n=1$ to the right-hand side of \eqref{connector:int1} without changing the equality when taking the limit.
Hence, by \eqref{hope}, \eqref{connector:int} and \eqref{connector:int1},
\begin{align}
&\quad\;\E_{z_0}^{\bi}\left[\int_0^{t}\d  L^{\bi}_{\tau}\e^{\beta_\bi \tau}K_0^{\bbeta,\bj}(\tau)\E^{\bj,\tau}_{\ms Z_{\tau}}\left[\e^{\beta_\bj (T_0^{\bj}-\tau)}F(\ms Z_{T_0^{\bj}},T_0^{
\bj});T_0^{\bj}\leq  t\right]
\right]\notag\\
&=\lim_{N_0\to\infty}\frac{w_\bi}{w_\bj}\sum_{n=1}^{2^{N_0}}
 \E_{z_0}^{\bbeta,\bs w,\bi}\Bigg[\frac{\e^{{A}^{\bbeta,\bs w}_0(T_0^{\bi,1})}}{2}F(\ms Z_{T_0^{\bi,1}},T_0^{\bi,1}); \left\{
 \begin{array}{cc}
 a_{n-1}<T_0^{\bi,1}=T_0^{\bj}\leq a_{n},\\
 Z^{\bi}_t\neq 0,\forall\;t\in [a_{n-1},T_0^{\bi,1}]
 \end{array}
 \right\}
 \Bigg].\label{connector:int2}
\end{align}

Finally, we show that 
\begin{align}
&\quad\lim_{N_0\to\infty} \frac{w_\bi}{w_\bj}\sum_{n=1}^{2^{N_0}}
 \E_{z_0}^{\bbeta,\bs w,\bi}\Bigg[\frac{\e^{{A}^{\bbeta,\bs w}_0(T_0^{\bi,1})}}{2}F(\ms Z_{T_0^{\bi,1}},T_0^{\bi,1}); \left\{
 \begin{array}{cc}
 a_{n-1}<T_0^{\bi,1}=T_0^{\bj}\leq a_{n},\\
 Z^{\bi}_t\neq 0,\forall\;t\in [a_{n-1},T_0^{\bi,1}]
 \end{array}
 \right\}
 \Bigg]\notag\\
 &=\frac{w_\bi}{w_\bj}
\E^{\bbeta,\bs w,\bi}_{z_0}\Bigg[\frac{\e^{{A}^{\bbeta,\bs w}_0(T_0^{\bi,1})}}{2}F(\ms Z_{T_0^{\bi,1}},T_0^{\bi,1});\left\{
\begin{array}{cc}
T_0^{\bi,1}\leq t,\\
\ms Z_{T_0^{\bi,1}}\in \CNwj
\end{array}
\right\}
\Bigg].\label{connector:int3}
\end{align}
Note that $\{(a_{n-1},a_n];1\leq n\leq 2^{N_0}\}=\{(a_{n-1}(N_0),a_n(N_0)];1\leq n\leq 2^{N_0}\}$ covers $(0,t]$,
$\lim_{N_0\to\infty}\max_{1\leq n\leq 2^{N_0}}(a_n-a_{n-1})=0$
by \eqref{def:an}, and $\{a_n(N_0);n\in \{0,\cdots, 2^{N_0}\}\}$ is increasing in $N_0$. Hence, the sequence of events 
\[
\bigcup_{n=1}^{2^{N_0}}\Big\{a_{n-1}<T_0^{\bi,1}=T_0^{\bj}\leq a_{n},\; Z^{\bi}_t\neq 0,\forall\;t\in [a_{n-1},T_0^{\bi,1}]\Big\},\quad N_0\in \Bbb N,
\]
is increasing in $N_0$ to 
\[
\left\{ 
0< T_0^{\bi,1}=T_0^{\bj}\leq t,\; \exists\;\delta\in (0,t)\mbox{ s.t. }Z^{\bi}_t\neq 0,\;\forall\;t\in (T_0^{\bi,1}-\delta,T_0^{\bi,1}]
\right\},
\]
which is $\P^{\bbeta,\bs w,\bi}_{z_0}$-a.s. equal to $\{0<T_0^{\bi,1}=T_0^{\bj}\leq t\}$  by the NTC of the stochastic many-$\delta$ motions \cite[Proposition~3.15 (7$\cc$)]{C:SDBG2-3}. Hence, \eqref{connector:int3} holds by monotone convergence. Combining \eqref{connector:int2} and \eqref{connector:int3} gives the required identity \eqref{connector:20}. The proof is complete.
\end{proof}

\end{document}